\documentclass[12pt]{article}%
\usepackage[latin1]{inputenc}
\usepackage{amsfonts,amssymb,latexsym}
\usepackage[Sonny]{fncychap}
\usepackage[T1]{fontenc}
\usepackage{amsbsy}
\usepackage{graphicx}
\usepackage{times}
\usepackage[latin1]{inputenc}
\usepackage[english]{babel}
\usepackage{amsmath,amsthm}
\usepackage{amsmath}
\usepackage{amsfonts}
\usepackage{amssymb}
\usepackage{mathrsfs}
\usepackage{color}
\usepackage[colorlinks=true, bookmarksnumbered=true, bookmarksopen=true,
bookmarksopenlevel=3, pdfstartview=FitH, linkcolor=blue, pdfmenubar=true,
pdftoolbar=true, bookmarks=true,citecolor=green, urlcolor=blue,
filecolor=magenta,plainpages=false,pdfpagelabels,breaklinks]{hyperref}%
\newtheorem{teor}{Theorem}[section]
\newtheorem{cor}[teor]{Corollary}
\newtheorem{lem}[teor]{Lemma}
\newtheorem{prop}[teor]{Proposition}
\theoremstyle{definition}
\newtheorem{defn}[teor]{Definition}
\theoremstyle{remark}
\newtheorem{rem}[teor]{Remark}

\numberwithin{equation}{section}

\begin{document}

\title{Gradient flows of time-dependent functionals in metric spaces and applications for PDEs \vspace{0.5cm}}
\author{{{Lucas C. F. Ferreira} {\thanks{L. Ferreira was supported by FAPESP and CNPQ,
Brazil. (corresponding author)}}}\\{\small Universidade Estadual de Campinas, Departamento de Matem\'{a}tica,}\\{\small {\ CEP 13083-859, Campinas-SP, Brazil.}}\\{\small \texttt{E-mail:lcff@ime.unicamp.br}}\vspace{0.5cm}\\{{Julio C. Valencia-Guevara }{\thanks{J. Valencia-Guevara was supported by CNPQ, Brazil.}}}\\{\small Universidade Estadual de Campinas, Departamento de Matem\'{a}tica,}\\{\small {\ CEP 13083-859, Campinas-SP, Brazil.}}\\{\small \texttt{E-mail:julioguevara08@gmail.com}}}
\date{}
\maketitle

\begin{abstract}
We develop a gradient-flow theory for time-dependent functionals defined in
abstract metric spaces. Global well-posedness and asymptotic behavior of
solutions are provided. Conditions on functionals and metric spaces allow to
consider the Wasserstein space $\mathscr{P}_{2}(\mathbb{R}^{d})$ and apply the
results for a large class of PDEs with time-dependent coefficients like
confinement and interaction potentials and diffusion. Our results can be seen
as an extension of those in Ambrosio-Gigli-Savar\'{e} (2005)\cite{Ambrosio} to the case of time-dependent
functionals. For that matter, we need to consider some residual terms,
time-versions of concepts like $\lambda$-convexity, time-differentiability of
minimizers for Moreau-Yosida approximations, and a priori estimates with
explicit time-dependence for De Giorgi interpolation. Here, functionals can be
unbounded from below and satisfy a type of $\lambda$-convexity that changes
as the time evolves.

\

\

 \vspace{0.1cm}


\end{abstract}

\

\

\section{Introduction}

\hspace{0.5cm}We consider the gradient flow equation%

\begin{align}
u^{\prime}(t)  &  =-\nabla\mathcal{E}(t,u(t)),\text{ }%
t>0,\label{eq:sist-grad-flow}\\
u(0)  &  =u_{0}, \label{eq:sist-grad-flow-initial-data}%
\end{align}
where $\mathcal{E}:[0,\infty)\times X\rightarrow(-\infty,\infty]$ is a
time-dependent functional and $(X,d)$ is a complete separable metric space.
Our aim is to construct a general theory in metric spaces that can be applied
for PDEs with time-dependent coefficients. In fact, with this theory in hand,
we obtain global-in-time existence and asymptotic behavior of solutions in the
Wasserstein space $\mathscr{P}_{2}(\Omega)$ for a number of PDEs with density of internal
energy $U$, confinement potential $V$ and interaction potential $W$ depending
on the time-variable. That space consists of probability measures on $\Omega$
with finite second moment endowed with the so-called Wasserstein metric
$\mathbf{d}_{2}(\mu,\nu)$. Here we will focus on the whole space
$\Omega=\mathbb{R}^{d}$.

Gradient flows theory has been successfully developed for the case of
time-independent functionals $\mathcal{E}(u)$ in general metric spaces $(X,d)$
(see \cite{De Giorgi-1},\cite{De Giorgi-2},\cite{Ambrosio},\cite{Ambrosio-2}%
,\cite{Marino-1}). Two basic tools in the theory are the concept of curves of
maximal slopes (see \cite{De Giorgi-2},\cite{Marino-1}) and a time-discrete
approximation scheme (see \cite{De Giorgi-1},\cite{Ambrosio-2}). The latter is
based on the implicit variational scheme
\begin{equation}
U_{\boldsymbol{\tau}}^{n}\in\underset{v\in X}{\text{Argmin}}\{\frac{1}{2\tau
}d^{2}(U_{\boldsymbol{\tau}}^{n-1},v)+\mathcal{E}(v)\}, \label{esquema-0}%
\end{equation}
where $\tau>0$ is a time step. Notice that (\ref{esquema-0}) consists in
finding minimizers for interactive values of the Moreau-Yosida approximation
$\mathscr{E}_{\tau}(u):=\inf_{v\in X}\{\frac{1}{2\tau}d^{2}(u,v)+\mathcal{E}%
(v)\}$ of $\mathcal{E}$ in $(X,d)$. Speak generally, basic hypotheses assumed
on $\mathcal{E}$ are lower semicontinuity and some type of convexity and
coercivity (see \cite{Ambrosio}). For the analysis of PDEs as a gradient
flows, a suitable metric space is $\mathscr{P}_{2}$ in which the above theory
has demonstrated to be particularly very fruitful. The idea of using the above
discrete scheme in $\mathscr{P}_{2}$ goes back to the work \cite{Jordan-Otto}
for the linear Fokker-Plank equation and \cite{Otto} for the porous medium
equation. Subsequently, several authors extended this approach to a general
class of continuity equations (see \cite{Ambrosio},\cite{Agueh}%
,\cite{Carrillo1}) with velocity field given by the gradient of the
variational derivative of a time-independent functional, namely
\begin{equation}
\frac{\partial u}{\partial t}=\mbox{div }\left(  u\nabla\frac{\delta
{\mathcal{E}}}{\delta u}\right),\; \mbox{in}\; (0,+\infty
)\times\mathbb{R}^{d}, \label{generalequations}%
\end{equation}
where $\mathcal{E}$ is the free energy associated
to PDE dealt with. Under some basic assumptions, they considered $\mathcal{E}$ with the form
\begin{equation}
{\mathcal{E}}[u]:=\int_{\mathbb{R}^{d}}U(u(x))\,dx+\int_{\mathbb{R}^{d}%
}u(x)\,V(x)\,dx+\frac{1}{2}\iint_{\mathbb{R}^{d}\times\mathbb{R}^{d}%
}W(x-y)\,u(x)\,u(y)\,dx\,dy, \label{generalfunctionals}%
\end{equation}
where $U:\mathbb{R}^{+}\rightarrow\mathbb{R}$ is the density of internal energy,
$V:\mathbb{R}^{d}\rightarrow\mathbb{R}$ is a confinement potential and
$W:\mathbb{R}^{d}\rightarrow\mathbb{R}$ is an interaction potential. The
functional (\ref{generalfunctionals}) has the classical form given by the sum
of the internal energy, potential energy and interaction energy functionals that is verified by
a wide number of physical models. Beside existence of global-in-time flows,
the literature contains results on uniqueness, global contraction, regularity,
and asymptotic stability of solutions (see e.g. \cite{Ambrosio}). We also
quote the paper \cite{Car-Fer-Pre} where a 1D non-local fluid mechanics model
with velocity coupled via Hilbert transform was analyzed by using gradient
flow theory in $\mathscr{P}_{2}$.

In \cite{Lisini}, the authors dealt with nonlinear diffusion equations in the
form
\[
\partial_{t}u-\text{div}(A(\nabla(f(u))+u\nabla V))=0,
\]
where $A$ is a symmetric matrix-valued function of the spatial variables
satisfying a uniform elliptic condition and $f$, $V$ are functions satisfying
suitable hypotheses. They also analyzed the contraction property for solutions.

On the other hand, from a theoretical and applied point of view, it is natural
to consider a time-dependence on the coefficients of some equations. For
instance, a version of the stochastic Fokker-Plank equation (the one
considered in \cite{Jordan-Otto}) is
\begin{equation}
dX_{t}=-\nabla V(t,X_{t})dt+\sqrt{2\kappa(t)}dB_{t}, \label{stoc-1}%
\end{equation}
where the term $\sqrt{2\kappa(t)}$ is known as the diffusion coefficient and
$B_{t}$ stands for the classical Brownian motion. For (\ref{stoc-1}), it is
well-known that the law of processes is modeled by the PDE
\[
\partial_{t}u=\kappa(t)\Delta u+\nabla\cdot(\nabla V(t,x)u).
\]
Another example is the version of the Mckean-Vlasov equation \cite{Veretennikov}
\[
dX_{t}=b(t,\mu_{t},X_{t})dt+\sqrt{2\kappa(t)}dB_{t},\text{ with }%
b(t,\mu,x)=-\nabla W(t,\cdot)\ast\mu,
\]
where $\mu_{t}$ is the law of the processes $X_{t}$ that obeys the PDE
\[
\partial_{t}u=\kappa\Delta u-\nabla\cdot(b(t,u,x)u)
\]
with $\kappa$ depending on the time $t$. The term $b(t,u,x)u$ corresponds to an
interaction between particles with time-dependent potential.

For a bounded convex domain $\Omega\subset\mathbb{R}^{d}$ and $0<T<\infty,$
Petrelli and Tudorascu \cite{Petrelli-Tudorasco} considered the
non-homogeneous Fokker-Plank equations
\begin{equation}
u_{t}-\nabla_{x}\cdot(u\nabla_{x}\psi(t,x))-\Delta_{x}(P(t,u))=g(t,x,u)\text{
in }\Omega\times(0,T) \label{aux-eq-1}%
\end{equation}
with Neumman boundary conditions and nonnegative $u_{0}\in L^{\infty}(\Omega)$
such that $\int u_{0}dx=1$. They proved existence of nonnegative bounded weak
solutions by constructing approximate solutions via time-interpolants of
minimizers arising from Wasserstein-type implicit schemes. \textbf{ }Let us
point out that, when $P(t,z)=\kappa(t)z$, the conditions in
\cite{Petrelli-Tudorasco} require that the viscosity $\kappa$ is bounded away
from zero, while here we allow $\kappa$ to be arbitrarily near zero (see
Theorem \ref{teor:kappa-time-depend} in subsection
\ref{subsect:variable-diffusion}).

In \cite{Savare}, Rossi, Mielke and Savar\'{e} analyzed the doubly nonlinear
evolution equation
\begin{equation}
\partial\psi(u^{\prime}(t))+\partial_{u}\mathcal{E}(t,u(t))\ni0\text{ in
}B^{\prime},\text{ a.e. }t\in(0,T), \label{aux-eq-2}%
\end{equation}
where $B$ is a separable Banach space, $0<T<\infty,$ and $u(0)=u_{0}.$ They
proposed a formulation for (\ref{aux-eq-2}) in a separable metric space
$(X,d)$ that extends the notion of curve of maximal slope for gradient flows
in metric spaces. Existence of solutions is proved by means of a time-discrete
approximation scheme in $(X,d)$ defined as
\begin{equation}
U_{\boldsymbol{\tau}}^{n}\in\underset{v\in X}{\text{Argmin}}\{\tau
\psi(d(U_{\boldsymbol{\tau}}^{n-1},v)/\tau)+\mathcal{E}(t_{n},v)\},
\label{esquema-1}%
\end{equation}
where $\boldsymbol{\tau}$ is a partition for $[0,T]$ and $\tau
=|\boldsymbol{\tau}|$ is the time step. Among others, the authors of
\cite{Savare} assumed that $\mathcal{E}$ satisfies the chain rule, is locally
(in time) uniformly bounded from below, and differentiable in the $t$-variable
with the derivative satisfying the condition%

\begin{equation}
|\partial_{t}\mathcal{E}(t,u)|\leq C(\mathcal{E}(t,u)+d(u^{\ast},u)+2C_{0}),
\label{aux-func-1}%
\end{equation}
for some $u^{\ast}\in X$, where
$$C_{0}=-\displaystyle\inf_{t\in\lbrack0,T],u\in X}\mathcal{E}(t,u).$$
In fact, functionals in \cite{Savare} are the sum of two
time-dependent functionals $\mathcal{E}_{1}$ and $\mathcal{E}_{2}$ where
$\mathcal{E}_{1}$ is bounded from below and $\lambda_{0}$-convex (uniformly
with respect to $t\in\lbrack0,T]$), and $\mathcal{E}_{2}$ is a dominated
concave perturbation of $\mathcal{E}_{1}.$ For a bounded domain $\Omega$ and
$u_{0}\in H_{0}^{1}(\Omega),$ using the above approach for $\partial
\psi(u^{\prime}(t))=u^{\prime}(t)$ (gradient flow case), they also analyzed
(\ref{eq:sist-grad-flow}) with $\partial_{u}\mathcal{E}(t,u)=-\Delta
u+F^{\prime}(u)-l(t)$ in the $L^{1}(\Omega)$-metric. These results were
improved in \cite{Mielke-Savare} by considering more general dissipation
$\psi$. Moreover, in \cite{Mielke-Savare} the condition (\ref{aux-func-1}) was
relaxed to $|\partial_{t}\mathcal{E}(t,u)|\leq C\mathcal{E}(t,u).$ We also
refer the reader to \cite{Mielke-Rossi-Savare-2, Roche-Rossi} for stability
results for doubly nonlinear equations in Banach spaces.

Since our functionals are not bounded from below neither satisfies a estimate
like (\ref{aux-func-1}), we can not to apply the theory from \cite{Savare} and
\cite{Mielke-Savare}. Here we assume the conditions \textbf{E1, E2, E3, E4
}and \textbf{E5 }given in Section \ref{set:preliminaries} (see pages 6 and 7).
Notice that \textbf{E4} gives some local-in-time control from below for
$\mathcal{E}$ but allows it to be unbounded from below at each $t>0$. In
(\ref{aux-func-1}) it is required some control of the time-derivative of
$\mathcal{E}$ in terms of the functional itself. Instead of such estimate, we
work with a condition on the difference of $\mathcal{E}$ in two different
times (see \textbf{E3}). In order to recover the contraction property,
inspired by the convexity used in \cite{Ambrosio}, we propose a type of
$\lambda$-convexity that changes as the time evolves (see \textbf{E5}). Thus,
functionals could \textquotedblleft lose convexity\textquotedblright\ in a
such way that the approximation between two solutions for large times still
holds, because the contraction property depends only on the mass accumulated
by $\lambda$, i.e. $\int_{0}^{t}\lambda(s)ds$. In general the function $\lambda(t)$ can be unbounded both from above and below in $[0,\infty)$ but, for the contraction, it is assumed to be continuous. In Section \ref{sect:applications-wasserstein}, we show how to extend
results for the case of ${\mathcal{E}}(t,u)$ having a more general density of internal
energy $U(t,u)$ and viscous term $-\Delta_{x}(P(t,u))$ (see Theorem
\ref{teor:internal energy} and Remark \ref{remark:flexibility} in subsection
\ref{sub-general}). There, the conditions on potentials prevent ${\mathcal{E}%
}(t,\rho)$ to satisfy \textbf{E3}. In the case $P(t,z)=\kappa(t)z$, the diffusion coefficient $\kappa$ is non-increasing. This condition is necessary in order to have the uniform limit of the approximate solutions (\ref{eq:sol-interpolante1}) in all finite interval
$[0,T]$.

Another application of time-dependent gradient flows appears in the context of
pursuit-evasion games. Jun \cite{Jun} considered gradient flows in suitable
playing fields and investigated existence and uniqueness of continuous pursuit
curves that are downward gradient curves for the distance from a moving
evader, i.e. a time-dependent gradient flows. In fact, his result works well
in $CAT(K)$-spaces (with $K=0$) that are complete metric spaces such that no
triangle is fatter than the triangle with same edge lengths in the model space
of constant curvature $K$. Also, he assumed that $\mathcal{E}(t,u)$ is
Lipschitz in $t,$ locally Lipschitz in $u,$ and $\lambda_{0}$-convex for all
$t>0$ where $\lambda_{0}$ is a fixed constant (i.e. $\lambda_{0}$-convex
uniformly in $t$). Another basic hypothesis used by him is that
$\mathscr{E}_{t,\tau}(u)$ given by
\[
\mathscr{E}_{t,\tau}(u):=\inf_{v\in X}\{\frac{1}{2\tau}d^{2}(u,v)+\mathcal{E}%
(t,v)\}
\]
is $C\tau$-Lipschitz in $t,$ for all $u\in X=$ $CAT(0)$ and $\tau>0,$ where
$C>0$ is a constant. For the time-independent case $\mathcal{E}(u)$, we refer
the reader to \cite{Mayer-1} for $X=$ $CAT(0)$ (see also \cite{Ambrosio}) and
\cite{Lytchak} for a geometric approach in\ $X=$ $CAT(K).$

In this paper we follow the program in the book \cite{Ambrosio} that contains
a relatively complete gradient flows theory in general metric spaces and its
applications for the non-vectorial space $\mathscr{P}_{2}$ by using optimal
transport tools. So, our results can be seen as an extension of those in
\cite{Ambrosio} in order to consider time-dependent functionals. For that
matter, due to time-dependence of $\mathcal{E}$, we need to handle some
residual terms (see e.g. (\ref{eq:residuo-G}) and the estimate
(\ref{eq:proposition-est-G})) and to consider time-versions of concepts like
$\lambda$-convexity (see \textbf{E5}) and interpolation functions as (\ref{eq:tempo-disc}) to
(\ref{eq:functional-interpolation}). One of these functions is the
interpolation (\ref{eq:lambda-discret}) that corresponds to the time-dependent
convexity parameter $\lambda(t)$. Thus, some adaptations from arguments in
\cite{Ambrosio} made here is not a straightforward matter and involves certain
care. Also, the time-differentiability of the minimizer for the Moreau-Yosida
approximation of $\mathcal{E}$ needs to be analyzed (see Proposition
\ref{prop:differentiability-property1}) and, in order to get the convergence
of the approximate solutions, a priori estimates with explicit dependence on the $t$-variable are performed in Proposition
\ref{prop:DiGiorgi-interp-converg} for which the aforementioned condition
\textbf{E3} plays a key role.

The plan of this paper is the following. In Section \ref{set:preliminaries} we
recall some concepts such as proper functional and local slope, and some
results on gradient flow theory in metric spaces. Also, we give the metric
formulation for
\eqref{eq:sist-grad-flow}-\eqref{eq:sist-grad-flow-initial-data} and the basic
assumptions for the functional $\mathcal{E}$. In Section
\ref{sect:construc-and-propert}, we construct the approximate solutions,
provide some properties for the minimizer of the Moreau-Yosida approximation,
and give estimates for approximate solutions. In Section
\ref{sect:apriori-estimates}, we derive \textit{a priori} estimates for the
approximate solutions and show their locally uniform convergence in
$[0,\infty)$. In Section \ref{sect:regularity}, we show that the curve, which
is the limit of the approximate solutions, is in fact a solution of
\eqref{eq:sist-grad-flow}-\eqref{eq:sist-grad-flow-initial-data} in the sense
of Section \ref{set:preliminaries} and obtain the contraction property for
solutions. Section \ref{sect:applications-wasserstein} is devoted to applying
the general theory in the Wasserstein space for PDEs with time-dependent
functionals as those mentioned above.

\section{Metric formulation and implicit scheme}

\label{set:preliminaries}\hspace{0.5cm}Let $(X,d)$ be a complete separable
metric space and consider the functional $\mathcal{E}:X\rightarrow(-\infty,+\infty]$. Recall that $\mathcal{E}$ is said to be proper whether there is
$u_{0}\in X$ such that $\mathcal{E}(u_{0})<\infty$, and its domain is defined by
\begin{equation}
\text{Dom}(\mathcal{E})=\{u\in X:\mathcal{E}(u)<\infty\}. \label{eq:domain}%
\end{equation}
Thus, a functional $\mathcal{E}$ is proper when $\text{Dom}(\mathcal{E}%
)\neq\emptyset$. Let $f^{+}$ and $f^{-}$ denote the positive and negative
parts of an extended real-valued function $f$. The following concept is
crucial in the theory of gradient flows, and we will use it for the case of
time-dependent functionals.

\begin{defn}
\label{defn:local-slope} Let $\mathcal{E}$ be a proper functional in a metric
space $X$. The \textit{local slope }$|\partial\mathcal{E}|$ of $\mathcal{E}$
at the point $u\in X$ is defined as
\begin{equation}
|\partial\mathcal{E}|(u)=\limsup_{v\rightarrow u}\frac{(\mathcal{E}%
(u)-\mathcal{E}(v))^{+}}{d(u,v)}. \label{eq:local-slope}%
\end{equation}

\end{defn}

In what follows, we recall a technical lemma that will be useful in our calculations.

\begin{lem}
[{\cite[Lemma 2.2.1]{Ambrosio}}]\label{lem:ambrosio-desig} Let $\mathcal{E}%
:X\rightarrow(\infty,\infty]$ be a functional such that there is $\tau^{\ast
}>0$ and $u^{\ast}\in X$ with
\[
\mathscr{E}_{\tau^{\ast}}(u^{\ast}):=\inf_{v\in X}\left\{  \mathcal{E}%
(v)+\frac{d^{2}(v,u^{\ast})}{2\tau^{\ast}}\right\}  >-\infty.
\]
Then
\[
\mathscr{E}_{\tau}(u)\geq\mathscr{E}_{\tau^{\ast}}(u^{\ast})-\frac{1}%
{\tau^{\ast}-\tau}d^{2}(u^{\ast},u),\text{ for all }0<\tau<\tau^{\ast}\text{
and }u\in X,
\]
and
\[
d^{2}(u,v)\leq\frac{4\tau^{\ast}\tau}{\tau^{\ast}-\tau}\left(  \mathcal{E}%
(v)+\frac{d^{2}(u,v)}{2\tau}-\mathscr{E}_{\tau^{\ast}}(u^{\ast})+\frac{1}%
{\tau^{\ast}-\tau}d^{2}(u^{\ast},u)\right)  .
\]
In particular, the sub-levels of the map $v\rightarrow\mathcal{E}%
(v)+\frac{d^{2}(u,v)}{2\tau}$ are bounded.
\end{lem}

\subsection{Metric formulation}

\label{subsect:metric-formulation}\hspace{0.5cm}Let $\mathcal{E}:[0,+\infty)\times
X\rightarrow(-\infty,+\infty]$ be a time-dependent functional. It is well known that the
problem (\ref{eq:sist-grad-flow})-(\ref{eq:sist-grad-flow-initial-data})
admits a metric reformulation by using the concept of local slope (see
\cite{Savare}). This is given by the variational inequality
\begin{equation}
\frac{d}{dt}(\mathcal{E}(t,u(t)))\leq\partial_{t}\mathcal{E}(t,u(t))-\frac
{1}{2}|\partial\mathcal{E}(t)|^{2}(u(t))-\frac{1}{2}|u^{\prime}|^{2}%
(t),\label{eq:variational-inequality}%
\end{equation}
where $|\partial\mathcal{E}(t)|$ stands for the local slope of the functional
$u\rightarrow\mathcal{E}(t,u)$, for each fixed $t>0$, and
\[
|u^{\prime}|(t)\lim_{s\rightarrow t}\frac{d(u(s),u(t))}{|t-s|}%
\]
stands for the metric derivative of an absolutely continuous curve $u$.

Below we state the principal assumptions on the family of functionals
$\mathcal{E}(t,\cdot)$ on $X$, for $t\in\lbrack0,\infty)$:

\begin{description}
\item[E1.-] For each $t\geq0$, $\mathcal{E}(t,\cdot)$ is proper and lower
semicontinuous with respect to the metric $d(\cdot,\cdot)$.

\item[E2.-] The domain of the functionals, $\mathbf{D}:=\text{Dom}%
(\mathcal{E}(t,\cdot))$, is time-independent.

\item[E3.-] There exist $u^{\ast}\in X$ and a function $\beta:[0,\infty
)\rightarrow\lbrack0,\infty)$ with $\beta\in L_{loc}^{1}([0,\infty))$ such
that, for each $u\in\mathbf{D}$, the function $t\rightarrow\mathcal{E}(t,u)$
satisfies
\begin{equation}
|\mathcal{E}(t,u)-\mathcal{E}(s,u)|\leq\int_{s}^{t}{\beta(r)\ dr}%
(1+d^{2}(u,u^{\ast})). \label{eq:E3-hypothesis}%
\end{equation}

\end{description}

Note that if the condition (\ref{eq:E3-hypothesis}) is valid for some
$u^{\ast}\in X$ then it is in fact valid for all $u^{\ast}\in X$. Also, for each $u\in\mathbf{D}$, the function $t\rightarrow\mathcal{E}(t,u)$ is differentiable a.e. in $[0,\infty)$ and its set of differentiability
points may depend on $u$.

Now we are ready to give the notion of solution for (\ref{eq:sist-grad-flow}%
)-(\ref{eq:sist-grad-flow-initial-data}) that we deal with.

\begin{defn}
\label{defn:def-solution} Let $u_{0}\in X$ and $\mathcal{E}:[0,+\infty)\times
X\rightarrow(-\infty,+\infty]$ be a functional satisfying the assumptions
\textbf{E1}, \textbf{E2} and \textbf{E3}. We say that an absolutely continuous
curve $u:[0,+\infty)\rightarrow X$ is a solution for (\ref{eq:sist-grad-flow}%
)-(\ref{eq:sist-grad-flow-initial-data}), if $u(0)=u_{0}$, the function
$t\rightarrow\mathcal{E}(t,u(t))$ is absolutely continuous,
\begin{equation}
|u^{\prime}|,\ |\partial\mathcal{E}(\cdot)|(u(\cdot))\in L_{loc}^{2}%
([0,\infty)), \label{eq:def-solution-1}%
\end{equation}
and the variational inequality (\ref{eq:variational-inequality}) holds true.
\end{defn}

\subsection{Implicit variational scheme}

\label{sucsect:euler-scheme}\hspace{0.5cm} We start by recalling the Moreau-Yosida approximation of $\mathcal{E}$. For $\tau>0$ and $t\geq0$, this approximation is defined as
\begin{equation}
\mathscr{E}_{t,\tau}(u):=\inf_{v\in X}\{\mathbf{E}(t,\tau,u;v)\},
\label{eq:moreau-yosida-approx1}%
\end{equation}
where the functional $\mathbf{E}(t,\tau,u;\cdot)$ is given by
\begin{equation}
\mathbf{E}(t,\tau,u;v):=\mathcal{E}(t,v)+\frac{d^{2}(u,v)}{2\tau}.
\label{eq:moreau-yosida-approx}%
\end{equation}

Next, take a partition $\boldsymbol{\tau}=\{0=t_{\boldsymbol{\tau}}^{0}<t_{\boldsymbol{\tau}}%
^{1}<\cdots<t_{\boldsymbol{\tau}}^{n}<\cdots\}$ of $[0,\infty)$ with
$\lim_{n\rightarrow\infty}t_{\boldsymbol{\tau}}^{n}=\infty$. Defining the
step size $\tau_{n}:=t_{\boldsymbol{\tau}}^{n}-t_{\boldsymbol{\tau}}^{n-1}$, one can construct the sequence
\begin{equation}
U_{\boldsymbol{\tau}}^{n}\in\underset{v\in X}{\text{Argmin}}\{\mathbf{E}%
(t_{\boldsymbol{\tau}}^{n},\tau_{n},U_{\boldsymbol{\tau}}^{n-1};v)\},
\label{eq:minimization-problem}%
\end{equation}
for a given family of initial data $U_{\boldsymbol{\tau}}^{0}\in X$.

Since the convergence results are locally in time, we can fix $T>0$ arbitrary and
analyze the convergence in $[0,T]$. In order to analyze rigorously the problem
of minimization (\ref{eq:minimization-problem}), we give two additional
assumptions that will allow to obtain uniqueness and a nice behavior of the minimizers.

\begin{description}
\item[E4.-] For each $T>0$, there exist a $u^{\ast}\in X$ and $\tau^{\ast
}(T)=\tau^{\ast}>0$ such that the function $t\rightarrow\mathscr{E}_{t,\tau
^{\ast}}(u^{\ast})$ is bounded from below in $[0,T]$.

\item[E5.-] \label{E5} There is a function $\lambda:[0,\infty)\rightarrow
\mathbb{R}$ in $L_{loc}^{\infty}([0,\infty))$ such that: given points
$u,v_{0},v_{1}\in X$, there exists a curve $\gamma:[0,1]\rightarrow X$
satisfying $\gamma(0)=v_{0}$, $\gamma(1)=v_{1}$ and
\begin{equation}
\mathbf{E}(t,\tau,u;\gamma(s))\leq(1-s)\mathbf{E}(t,\tau,u;v_{0}%
)+s\mathbf{E}(t,\tau,u;v_{1})-\frac{1+\tau\lambda(t)}{2\tau}s(1-s)d^{2}%
(v_{0}.v_{1}), \label{eq:convexity}%
\end{equation}
for $0<\tau<\frac{1}{\lambda_{T}^{-}}$ and $s\in\lbrack0,1]$, where
$\lambda_{T}^{-}=\max\{-\inf_{t\in[0,T]}\lambda(t),0\}$.
\end{description}

\begin{rem}
\label{Rem-aux-1} Note that by Lemma \ref{lem:ambrosio-desig}, for each
$0<\tau<\tau^{\ast}$ and $u\in X$, we have that the function $t\rightarrow
\mathscr{E}_{t,\tau}(u)$ is bounded from below in $[0,T]$. In view of the
assumptions \textbf{E4} and \textbf{E5} we assume by technical reasons that
$\tau^{\ast}<\min\{\frac{1}{\lambda_{T+1}^{-}},1\}$.
\end{rem}

\begin{rem}
\label{rem:aux-2} In \textbf{E5}, we consider the existence of curves $\gamma
:[0,1]\rightarrow X$ for all $v_{0},v_{1}\in X$ and not only for elements in
the domain $\mathbf{D}$. This will be necessary for the applications in Section
\ref{sect:applications-wasserstein} where we will use the concept of generalized
geodesics in the Wasserstein space $\mathscr{P}_{2}(\mathbb{R}^{d}).$ These
curves exist independently of the functionals that we will analyze in that section.
\end{rem}

\section{Construction and properties of the implicit scheme}

\label{sect:construc-and-propert}\hspace{0.5cm}In this section we provide some
results about the sequence defined in (\ref{eq:minimization-problem}). They
can be seen as extensions of some results in \cite{Ambrosio} to the case of
time-dependent functionals. We start with the following preliminary result.

\begin{lem}
\label{lemma:existence-minimizer} Suppose \textbf{E1}, \textbf{E4} and
\textbf{E5} and let $u\in X$, $0\leq t\leq T$, $0<\tau<\frac{1}{\lambda
_{T}^{-}}$. Then, the minimization problem
\[
\min_{v\in X}\{\mathbf{E}(t,\tau,u;v)\}
\]
has a unique minimizer $u_{\tau}^{t}$.
\end{lem}

\begin{proof}
Let $v_n\in \mathbf{D}$ be a minimizing sequence, i.e., $\displaystyle\lim_{n\to\infty}\mathbf{E}(t,\tau,u;v_n)=
\mathscr{E}_{t,\tau}(u)$. Given $m,n\in\mathbb{N}$, by the convexity property {\bf E5}, there is a curve $\gamma:[0,1]\to X$,
with $\gamma(0)=v_n$, $\gamma(1)=v_m$, and
\begin{eqnarray*}
\mathscr{E}_{t,\tau}(u)&\leq&\mathbf{E}(t,\tau,u;\gamma(1/2))\\
&\leq&\frac{1}{2}\mathbf{E}(t,\tau,u;v_n)+\frac{1}{2}\mathbf{E}(t,\tau,u;v_m)-\frac{\tau^{-1}+\lambda(t)}{8}d^2(v_n,v_m).
\end{eqnarray*}
Thus
\begin{eqnarray*}
\frac{\tau^{-1}+\lambda(t)}{4}d^2(v_n,v_m)&\leq&\left(\mathbf{E}(t,\tau,u;v_n)-\mathscr{E}_
{t,\tau}(u)\right)+\left(\mathbf{E}(t,\tau,u;v_m)-\mathscr{E}_{t,\tau}(u)\right).
\end{eqnarray*}
It follows from the above estimative that $v_n$ is a Cauchy sequence in $X$ and hence it converges to some $u_{\tau}^t\in X$. From the lower semicontinuity, we get that $u_{\tau}^t$ is a minimizer of the functional $\mathbf{E}(t,\tau,u;\cdot)$. The uniqueness follows from {\bf E5} and is left to the reader.
\end{proof}

In the next lemma we show that $\mathscr{E}_{\tau}^{t}(u)$ and $u_{\tau}^{t}$
depend continuously on $(\tau,t,u).$

\begin{lem}
\label{lem:continuity-minimizer} Assume the properties \textbf{E1} to
\textbf{E5}. Then, the following statements hold true:

\begin{description}
\item[a)] The map $(\tau,t,u)\in(0,\tau^{\ast}(T))\times\lbrack0,T]\times
X\rightarrow\mathscr{E}_{\tau}^{t}(u)\in\mathbb{R}$ is continuous.

\item[b)] The map $(\tau,t,u)\in(0,\tau^{\ast}(T))\times\lbrack0,T]\times X\to
u_{\tau}^{t}\in X$ is continuous.
\end{description}
\end{lem}

\begin{proof}
We start with item a). Let $(\tau_n,t_n,u_n)$ be a sequence converging to $(\tau_0,t_0,u_0)$ in $(0,\tau^{\ast}(T))\times\lbrack0,T]\times X$. Denote
by $v_n=(u_n)_{\tau_n}^{t_n}$ the minimizer of $\mathbf{E}(t_n,\tau_n,u_n;\cdot)$ given in Lemma \ref{lemma:existence-minimizer}. It follows that
\begin{eqnarray*}
\limsup_{n\to\infty}\mathscr{E}_{t_n,\tau_n}(u_n)&=&\limsup_{n\to\infty}\mathbf{E}(t_n,\tau_n,u_n;v_n)\\
&\leq&\limsup_{n\to\infty}\mathbf{E}(t_n,\tau_n,u_n;v)\\
&=&\mathbf{E}(t_0,\tau_0,u_0;v),
\end{eqnarray*}
for all $v\in X$. Taking the infimum in the right hand side, we obtain $\displaystyle\limsup_{n\to\infty}\mathscr{E}_{t_n,
\tau_n}(u_n)\leq\mathscr{E}_{t_0,\tau_0}(u_0)$. In view of Lemma \ref{lem:ambrosio-desig}, the sequence $v_n$ is bounded. So, we can estimate
\begin{eqnarray*}
\liminf_{n\to\infty}\mathscr{E}_{t_n,\tau_n}(u_n)&\geq&\liminf_{n\to\infty}\left\{\frac{1}{2\tau_n}(d(u_n,u_0)-
d(v_n,u_0))^2 +\mathcal{E}(t_n,v_n)\right\}\\
&=&\liminf\left\{\frac{d^2(u_n,u_0)}{2\tau_n}-\frac{d(u_n,u_0)d(u_0,v_n)}{\tau_n}+
\frac{\tau_0-\tau_n}{2\tau_0\tau_n}d^2(v_n,u_0)\right.\\
&\ &\left. +\mathbf{E}(t_0,\tau_0,u_0;v_n)+
(\mathcal{E}(t_n,v_n)-\mathcal{E}(t_0,v_n))\right\}\\
&\geq& \liminf_{n\to\infty}\left\{\mathscr{E}_{t_0,\tau_0}(u_0)-\int_{[t_n,t_0]}{\beta(r)\ dr}
(1+d^2(u^*,v_n))\right\}\\
&=&\mathscr{E}_{t_0,\tau_0}(u_0),
\end{eqnarray*}
and thus $\displaystyle\lim_{n\to\infty}\mathscr{E}_{t_n,\tau_n}(u_n)=\mathscr{E}_{t_0,\tau_0}(u_0)$. For item b),
note that
\begin{eqnarray*}
\mathbf{E}(t_0,\tau_0,u_0;v_n)-\mathscr{E}_{t_n,\tau_n}(u_n)&=&\left(\frac{d^2(u_0,v_n)}{\tau_0}
-\frac{d^2(v_n,u_n)}{\tau_n}\right)\\
&\ &+(\mathcal{E}(t_0,v_n)-\mathcal{E}(t_n,v_n)).
\end{eqnarray*}
Using \textbf{E3}, the boundedness of $v_n$ and the convergence $(\tau_n,t_n,u_n)\rightarrow (\tau_0,t_0,u_0)$,
we get $$\lim_{n\to \infty}(\mathbf{E}(t_0,\tau_0,u_0;v_n)-\mathscr{E}_{t_n,\tau_n}(u_n))=0,$$
and, by the item a),
$$\lim_{n\to \infty}\mathbf{E}(t_0,\tau_0,u_0;v_n)=\mathscr{E}_{t_0,\tau_0}(u_0).$$
It follows that $v_n$ is also a minimizing sequence for $\mathscr{E}_{t_0,\tau_0}(u_0)$ and then, by the same arguments in the proof of Lemma \ref{lemma:existence-minimizer}, it converges to $(u_0)_{\tau_0}^{t_0}$, as required.
\end{proof}

Because of Lemma \ref{lemma:existence-minimizer}, for each family of initial data $U_{\boldsymbol{\tau}}^{0}\in X$
associated to a partition $\boldsymbol{\tau}$ of $[0,+\infty),$ we have that
the sequence (\ref{eq:minimization-problem}) is well-defined for each
$n\in\mathbb{N}$ such that $t_{\boldsymbol{\tau} }^{n}<T+\tau^{*}$, if
$\tau^{*}<\frac{1}{\lambda_{T+1}^{-}}$. In what follows, we give some
estimates for the minimizer of the Moreau-Yosida approximation \eqref{eq:moreau-yosida-approx1}. These will
play an important role in the convergence of approximate solutions.

\begin{lem}
\label{lem:minimizer-estimative1} Assume that $\mathcal{E}$ satisfies the
properties \textbf{E1} to \textbf{E5}. Let $0<\tau<\tau^{\ast}$, $0\leq t \leq
T$ and $u\in\mathbf{D}$. Then
\begin{equation}
d^{2}(u_{\tau}^{t+\tau},u^{\ast})-d^{2}(u,u^{\ast})\leq\epsilon\frac
{d^{2}(u_{\tau}^{t+\tau},u)}{\tau}+\tau\frac{d^{2}(u_{\tau}^{t+\tau},u^{\ast
})}{\epsilon}, \label{eq:estimative-aux1}%
\end{equation}
for all $\epsilon>0$. If $\tau\leq\tau^{\ast}/8$, we have
\begin{align}
d^{2}(u_{\tau}^{t+\tau},u^{\ast})  &  \leq4\tau^{\ast}\left(  \mathcal{E}%
(t,u)+\int_{t}^{t+\tau^{\ast}}{\beta(r)\ dr}(1+d^{2}(u,u^{\ast}))-\inf_{0\leq
r\leq t+\tau^{\ast}}\mathscr{E}_{r,\tau^{\ast}}(u^{\ast}))\right) \nonumber\\
&  +4d^{2}(u,u^{\ast}). \label{eq:estimative-aux2}%
\end{align}

\end{lem}

\begin{proof} We have that
\begin{eqnarray}
d^2(u^{t+\tau}_{\tau},u^*)-d^2(u,u^*)&=&-2d(u^{t+\tau}_{\tau},u^*)(d(u,u^*)-d(u^{t+\tau}_{\tau},u^*))\nonumber\\
&\ &-(d(u,u^*)-d(u^{t+\tau}_{\tau},u^*))^2\nonumber\\
&\leq&2d(u^{t+\tau}_{\tau},u^*)d(u^{t+\tau}_{\tau},u)\nonumber\\
&\leq&\epsilon\frac{d^2(u^{t+\tau}_{\tau},u)}{\tau}+\tau\frac{d^2(u^{t+\tau}_{\tau},u^*)}
{\epsilon}\label{eq:est-aux}\\
&\leq&2\epsilon(\mathcal{E}(t+\tau,u)-\mathcal{E}(t+\tau,u^{t+\tau}_{\tau}))+\tau\frac{d^2(u^{t+\tau}_{\tau},u^*)}
{\epsilon}\nonumber\\
&\leq&2\epsilon(\mathcal{E}(t+\tau,u)-\mathscr{E}_{t+\tau,\tau^*}(u^*))\nonumber\\
&\ &+\epsilon\frac{d^2(u^{t+\tau}_{\tau},u^*)}{\tau^*}
+\tau\frac{d^2(u^{t+\tau}_{\tau},u^*)}{\epsilon}.\label{aux-est-1}
\end{eqnarray}
Notice that we have already obtained (\ref{eq:estimative-aux1}) in (\ref{eq:est-aux}). Now, choosing $\epsilon=\frac{\tau^*}{2}$ in (\ref{aux-est-1}), we get
\begin{eqnarray*}
\frac{1}{2}d^2(u^{t+\tau}_{\tau},u^*)&\leq&\tau^*\left(\mathcal{E}(t,u)+\int_t^{t+\tau^*}{\beta(r)\
dr}(1+d^2(u,u^*))-\mathscr{E}_{t+\tau,\tau^*}(u^*))\right)\\
&\ &+d^2(u,u^*)+\frac{2\tau}{\tau^*}d^2(u^{t+\tau}_{\tau},u^*),
\end{eqnarray*}
which implies (\ref{eq:estimative-aux2}) when $\tau\leq\tau^*/8$.
\end{proof}The next result gives a time-differentiability property for
$\mathscr{E}_{t,\tau}(u)$.

\begin{prop}
\label{prop:differentiability-property1} Assume \textbf{E1} to
\textbf{E5}. For $0<\tau\leq\frac{\tau^{\ast}}{8}$, the function
$\tau\rightarrow\mathscr{E}_{t+\tau,\tau}(u)$ is locally absolutely continuous
in $(0,\frac{\tau^{\ast}}{8}]$ and then is differentiable almost everywhere
in that interval. For each $u\in\mathbf{D,}$ assume further that the set of
differentiability points of $t\rightarrow\mathcal{E}(t,u)$ does not depend on
$u$ (e.g., when $t\rightarrow\mathcal{E}(t,u)$ is differentiable). Then
\begin{equation}
\frac{d}{d\tau}\mathscr{E}_{t+\tau,\tau}(u)=\partial_{t}\mathcal{E}%
(t+\tau,u_{\tau}^{t+\tau})-\frac{d^{2}(u,u_{\tau}^{t+\tau})}{2\tau^{2}}
\label{eq:formula-of-derivative}%
\end{equation}
in the set of differentiability points.
\end{prop}

\begin{proof}
Let $0<\tau_0<\tau_1\leq\frac{\tau^*}{8}$. Recalling that $u_{\tau_0}^{t+\tau_0}$ minimizes $E(t+\tau_0,\tau_0, u; \cdot)$, and using \textbf{E3}, we have that
\begin{eqnarray}
\mathscr{E}_{t+\tau_1,\tau_1}(u)-\mathscr{E}_{t+\tau_0,\tau_0}(u)&\leq&\mathcal{E}(t+\tau_1,u_{\tau_0}^{t+\tau_0})
-\mathcal{E}(t+\tau_0,u_{\tau_0}^{t+\tau_0})\nonumber\\
&\ &+\frac{\tau_0-\tau_1}{2\tau_1\tau_0}d^2(u,u_{\tau_0}^{t+\tau_0})\label{eq:differentiability-property-aux1}\\
&\leq&\int_{t+\tau_0}^{t+\tau_1}{\beta(r)\ dr}(1+d^2(u^{t+\tau_0}_{\tau_0},u^*))+
\frac{\tau_0-\tau_1}{2\tau_1\tau_0}d^2(u,u_{\tau_0}^{t+\tau_0}).\nonumber
\end{eqnarray}
Similarly, but now using $u_{\tau_1}^{t+\tau_1}$, it follows that
\begin{eqnarray}
\mathscr{E}_{t+\tau_1,\tau_1}(u)-\mathscr{E}_{t+\tau_0,\tau_0}(u)&\geq&\mathcal{E}(t+\tau_1,u_{\tau_1}^{t+\tau_1})
-\mathcal{E}(t+\tau_0,u_{\tau_1}^{t+\tau_1})\nonumber\\
&\ &+\frac{\tau_0-\tau_1}{2\tau_1\tau_0}d^2(u,u_{\tau_1}^{t+\tau_1})\label{eq:differentiability-property-aux2}\\
&\geq&\frac{\tau_0-\tau_1}{2\tau_1\tau_0}d^2(u,u_{\tau_1}^{t+\tau_1})
-\int_{t+\tau_0}^{t+\tau_1}{\beta(r)\ dr}(1+d^2(u^{t+\tau_1}_{\tau_1},u^*)).\nonumber
\end{eqnarray}
Notice that (\ref{eq:estimative-aux2}) allows us to estimate
the terms $d(u^*,u^{t+\tau_i}_{\tau_i})$ and $d(u,u^{t+\tau_i}_{\tau_i})$, $i=0,1$, by an expression independent of $\tau$, which gives the absolute continuity in each compact interval of
$(0,\frac{\tau^*}{8}]$. Now take a point $\tau\in(0,\frac{\tau^*}{8}]$ where the derivative of $\tau\to\mathscr{E}_{t+\tau,\tau}$ exists. Considering lateral limits, the equality (\ref{eq:formula-of-derivative}) follows by using estimates (\ref{eq:differentiability-property-aux1})-(\ref{eq:differentiability-property-aux2}) and that $d(u_{\tau}^{t+\tau},u_{\tau_k}^{t+\tau_k})\to0$ as $\tau_k\rightarrow\tau$ (see Lemma \ref{lem:continuity-minimizer} b)).
\end{proof}

As a consequence, we have the following corollary.

\begin{cor}
\label{cor:integral-identity} Assume the same hypotheses of Proposition
\ref{prop:differentiability-property1}. Then, for $u\in\mathbf{D}$, we have
the identity
\begin{equation}
\label{eq:integral-identity}\frac{d^{2}(u,u_{\tau}^{t+\tau})}{2\tau}+\int
_{0}^{\tau}{\frac{d^{2}(u,u_{r}^{t+r})}{2r^{2}}\ dr}=\int_{0}^{\tau}%
{\partial_{t} \mathcal{E}(t+r,u_{r}^{t+r})\ dr}+\mathcal{E}(t,u)-\mathcal{E}%
(t+\tau,u_{\tau}^{t+\tau}).
\end{equation}

\end{cor}

\begin{proof}
By integrating (\ref{eq:formula-of-derivative}) from $\tau_0$ to $\tau\leq\frac{\tau^*}{8}$, it follows that
\begin{eqnarray*}
\mathscr{E}_{t+\tau,\tau}(u)-\mathscr{E}_{t+\tau_0,\tau_0}(u)+\int_{\tau_0}^{\tau}{\frac{d^2(u,u_{r}^{t+r})}{2r^2}\
dr}=\int_{\tau_0}^{\tau}{\partial_t\mathcal{E}(t+r,u_{r}^{t+r})\ dr}.
\end{eqnarray*}
In view of the definitions of $\mathscr{E}_{t,\tau}(u)$ and $u_{\tau}^{t}$, and since the above integrals are finite as $\tau_0\to0$, the remainder of the proof is to show that $\mathscr{E}_{t+\tau_0,\tau_0}(u)\to
\mathcal{E}(t,u)$ as $\tau_0\to0$, for each fixed $t>0$. In fact, note that
$$\mathscr{E}_{t+\tau_0,\tau_0}(u)\leq\mathcal{E}(t+\tau_0,u),$$
and so
\begin{equation}
\limsup_{\tau_0\to0^+}\mathscr{E}_{t+\tau_0,\tau_0}(u)\leq\mathcal{E}(t,u).\label{aux-gen-10}
\end{equation}
Also, we can conclude from \eqref{aux-gen-10} and Lemma \ref{lem:ambrosio-desig} that $d(u,u_{\tau_0}^{t+\tau_0})\to0$, as $\tau_0
\to0$. Using the lower semicontinuity of $\mathcal{E}$, we get
\begin{eqnarray*}
\mathcal{E}(t,u)&\geq&\limsup_{\tau_0\to0}\mathscr{E}_{t+\tau_0,\tau_0}(u)\\
&\geq&\liminf_{\tau_0\to0}\left(\mathcal{E}(t,u_{\tau_0}^{t+\tau_0})-\int_t^{t+\tau_0}{\beta(r)\ dr}(1+d^2(u^*,
u_{\tau_0}^{t+\tau_0}))\right)\\
&\geq&\mathcal{E}(t,u),
\end{eqnarray*}
as desired.
\end{proof}

\begin{rem}
\label{remark:continuity-in-zero-of-minimizer} In the last proof, we have
showed in particular that $u_{\tau}^{t+\tau}\to u$ as $\tau\to0$, when
$u\in\mathbf{D}$.
\end{rem}

Now, we recall a discrete Gronwall lemma.

\begin{lem}
[\cite{Ambrosio} Lemma 3.2.4]\label{lem:discrete-gronwall} Let $A,\alpha
\in[0,\infty)$ and, for $n\geq1$, let $a_{n},\beta_{n}\in[0,\infty)$ satisfy
\[
a_{n}\leq A+\alpha\sum_{j=1}^{n}\beta_{j}a_{j},\ \forall n\geq1,\ \text{with
}m=\sup_{n\in\mathbb{N}}\alpha\beta_{n}<1.
\]
Then, denoting $B:=A/(1-m)$, $\theta:=\alpha/(1-m)$ and $\beta_{0}=0$, we
have that
\[
a_{n}\leq B e^{\theta\sum_{i=0}^{n-1}{\beta_{i}}},\ n\geq1.
\]

\end{lem}

The variational scheme (\ref{eq:minimization-problem}) will be the base for
constructing approximate solutions for (\ref{eq:sist-grad-flow}%
)-(\ref{eq:sist-grad-flow-initial-data}). The below lemma can be seen as a
version of \cite[Lemma 3.2.2]{Ambrosio} for the case of time-dependent
functionals and gives a first set of estimates in order to control approximations.

\begin{lem}
\label{lemma:limit-sec-discret} Assume \textbf{E1} to \textbf{E5}. Let
$\boldsymbol{\tau}=\{0=t_{\boldsymbol{\tau}}^{0}<t_{\boldsymbol{\tau}}%
^{1}<\cdots<t_{\boldsymbol{\tau}}^{j}<\cdots\}$ be a partition of $[0,\infty
)$, $\tau_{j}:=t_{\boldsymbol{\tau}}^{j}-t_{\boldsymbol{\tau}}^{j-1},$ and
$|\boldsymbol{\tau}|=\sup_{j}\left\vert \tau_{j}\right\vert $. For $T>0$ and
$\tau^{*}<\frac{1}{\lambda_{T+ 1}^{-}}\,$, choose $N\in\mathbb{N}$ such that
$T\in\lbrack t_{\boldsymbol{\tau}}^{N-1},t_{\boldsymbol{\tau}}^{N})$. Suppose
that there is a constant $S>0$ satisfying
\begin{equation}
\mathcal{E}(0,U_{\boldsymbol{\tau}}^{0})\leq S\,\ \text{and }d^{2}(u^{\ast
},U_{\boldsymbol{\tau}}^{0})\leq S. \label{eq:limit-sec-discret}%
\end{equation}
Then, there exists a constant $C=C(S,T,\tau^{\ast},\mathcal{E})>0$ such that
\begin{equation}
d^{2}(u^{\ast},U_{\boldsymbol{\tau}}^{n})\leq C,\hspace{0.5cm}\sum_{j=1}%
^{n}{\frac{d^{2}(U_{\boldsymbol{\tau}}^{j},U_{\boldsymbol{\tau}}^{j-1})}%
{2\tau_{j}}}\leq\sum_{j=1}^{n}{\left(  \mathcal{E}(t_{\boldsymbol{\tau}}%
^{j},U_{\boldsymbol{\tau}}^{j-1})-\mathcal{E}(t_{\boldsymbol{\tau}}%
^{j},U_{\boldsymbol{\tau}}^{j})\right)  }\leq C, \label{eq:limit-sec-discret1}%
\end{equation}
for all $1\leq n\leq N$ and $|\boldsymbol{\tau}|$ sufficiently small.
\end{lem}

\begin{proof}
By the minimizer property of $U^j_{\boldsymbol{\tau}}$ and {\bf E3}, we get
\begin{eqnarray}
\sum_{j=1}^{n}{\frac{d^2(U_{\boldsymbol{\tau}}^{j},U_{\boldsymbol{\tau}}^{j-1})}{2\tau_j}}
&\leq&\sum_{j=1}^{n}{\left(  \mathcal{E}(t_{\boldsymbol{\tau}}%
^{j},U_{\boldsymbol{\tau}}^{j-1})-\mathcal{E}(t_{\boldsymbol{\tau}}%
^{j},U_{\boldsymbol{\tau}}^{j})\right)}\nonumber\\
&\leq&\mathcal{E}(0,U_{\boldsymbol{\tau}}^{0})-\mathcal{E}(t_{\boldsymbol{\tau}}^{n},U_{\boldsymbol{\tau}}^{n})\nonumber\\
&\ &+\sum_{j=1}^n{\int_{t_{\boldsymbol{\tau}}^{j-1}}^{t_{\boldsymbol{\tau}}^{j}}{\beta(r)\ dr(1+d^2(u^*,U_
{\boldsymbol{\tau}}^{j-1}))}},\label{eq:limit-sec-discret2}
\end{eqnarray}
for $1\leq n\leq N$. Using the first estimate in Lemma \ref{lem:minimizer-estimative1} with $u=U
_{\boldsymbol{\tau}}^{j-1}$, $u_{\tau}^{t+\tau}=U_{\boldsymbol{\tau}}^{j}$ and $\epsilon=\frac{\tau^*}{2}$, we obtain
\begin{eqnarray*}
\frac{1}{2}d^2(u^*,U_{\boldsymbol{\tau}}^{n})-\frac{1}{2}d^2(u^*,U_{\boldsymbol{\tau}}^{0})&=&
\sum_{j=1}^n{\frac{1}{2}d^2(u^*,U_{\boldsymbol{\tau}}^{j})-\frac{1}{2}d^2(u^*,U_{\boldsymbol{\tau}}^{j-1})}\\
&\leq&\frac{\tau^*}{2}\left(\mathcal{E}(0,U_{\boldsymbol{\tau}}^{0})-\inf_{0\leq t\leq T+\tau^*}\mathscr{E}_{t,\tau^*}(u^*)
\right)+\frac{d^2(u^*,U_{\boldsymbol{\tau}}^{n})}{4}\\
&\ &+\sum_{j=1}^n{\tau_j\frac{d^2(U_{\boldsymbol{\tau}}^{j},u^*)}{\tau^*}
+\left(\frac{\tau^*}{2}\int_{t_{\boldsymbol{\tau}}^{j-1}}^{t_{\boldsymbol{\tau}}^{j}}{\beta(r)\ dr}\right){d^2(u^*,U_{\boldsymbol{\tau}}^{j-1})}}\\
&\ &+\frac{\tau^*}{2}\int_{0}^{T+\tau^*}{\beta(r)\ dr}.
\end{eqnarray*}
Rearranging terms, it follows that
\begin{eqnarray}
d^2(u^*,U_{\boldsymbol{\tau}}^{n})&\leq&2\tau^*\left(S-
\inf_{0\leq t\leq T+\tau^*}\mathscr{E}_{t,\tau^*}(u^*)\right)+2\left(1+\tau^*\int_
{0}^{\tau^*}{\beta(r)\ dr}\right)S\nonumber\\
&\ &+2\tau^*\int_{0}^{T+\tau^*}{\beta(r)\ dr}
+4\sum_{j=1}^{n}{\left(\frac{\tau_j}{\tau^*}+\frac{\tau^*}{2}\int_
{t_{\boldsymbol{\tau}}^{j}}^{t_{\boldsymbol{\tau}}^{j+1}}{\beta(r)\ dr}\right)d^2(u^*,U_{\boldsymbol{\tau}}^{j})}\nonumber\\
&\leq& A(S,T,\tau^*,\mathcal{E})+4\sum_{j=1}^{n}{\beta_j d^2(u^*,U_{\boldsymbol{\tau}}^{j})}, \label{aux-est-2}
\end{eqnarray}
for some constant $A=A(S,T,\tau^*,\mathcal{E})>0$, where $\beta_j:=\frac{\tau_j}{\tau^*}+\frac{\tau^*}{2}\int_
{t_{\boldsymbol{\tau}}^{j}}^{t_{\boldsymbol{\tau}}^{j+1}}{\beta(r)\ dr}$. By using an argument of absolute continuity, we have that $\displaystyle\max_{1\leq n\leq N}{4\beta_j}<1$, for $\vert\boldsymbol{\tau}\vert$ small enough. Then, the first estimate in (\ref{eq:limit-sec-discret1}) follows by using Lemma \ref{lem:discrete-gronwall} in (\ref{aux-est-2}). For the second one, we use (\ref{eq:limit-sec-discret2}) and observe that
\begin{eqnarray*}
\mathcal{E}(0,U_{\boldsymbol{\tau}}^{0})-\mathcal{E}(t_{\boldsymbol{\tau}}^{n},U_{\boldsymbol{\tau}}^{n})&\leq&
S-\inf_{0\leq t\leq T+\tau^*}\mathscr{E}_{t,\tau^*}(u^*)+\frac{d^2(u^*,U_{\boldsymbol{\tau}}^{n})}{2\tau^*},
\end{eqnarray*}
which is bounded. This concludes the proof.
\end{proof}

\section{A priori estimates}

\label{sect:apriori-estimates}\hspace{0.5cm}It is well known that, under
convexity hypotheses, the problem (\ref{eq:sist-grad-flow}%
)-(\ref{eq:sist-grad-flow-initial-data}) admits a formulation based in a
differential inequality. In fact, in the case when $X$ is a Euclidean space
and the functional $\mathcal{E}(t,\cdot)$ is $\lambda(t)-$convex, the curve
solution $u(t)$ satisfies
\begin{equation}
\frac{1}{2}\frac{d}{dt}\Vert u(t)-v\Vert^{2}+\frac{\lambda(t)}{2}\Vert
u(t)-v\Vert^{2}+\mathcal{E}(t,u(t))\leq\mathcal{E}(t,v),
\label{eq:formulation-convex}%
\end{equation}
for all $v\in X$. Assuming the hypothesis of convexity \textbf{E5}$,$ one can
derive a discrete version of (\ref{eq:formulation-convex}). In fact, for each
fixed $t>0,$ we have (see \cite[Theorem 4.1.2]{Ambrosio})
\begin{equation}
\frac{1}{2\tau}d^{2}(u_{\tau}^{t},v)-\frac{1}{2\tau}d^{2}(u,v)+\frac{1}%
{2}\lambda(t)d^{2}(u_{\tau}^{t},v)\leq\mathcal{E}(t,v)-\mathscr{E}_{t,\tau
}(u). \label{eq:formulation-convex-discret}%
\end{equation}

Now we define a set of interpolating functions that will be useful in the
convergence of approximate solutions. In comparison with \cite{Ambrosio}, the time-dependence of
$\mathcal{E}$ generates new residual terms in the estimates and leads us to
define the interpolations $\mathcal{T}_{\boldsymbol{\tau}}$ and $\widetilde
{\lambda}_{\boldsymbol{\tau}}(t)$ in \eqref{eq:tempo-disc}-\eqref{eq:lambda-discret} below. The function $\widetilde
{\lambda}_{\boldsymbol{\tau}}(t)$ is necessary in order to deal with the time-dependence on the parameter $\lambda$.

Let $\boldsymbol{\tau}=\{0=t_{\boldsymbol{\tau}}^{0}<t_{\boldsymbol{\tau}}%
^{1}<\cdots<t_{\boldsymbol{\tau}}^{n}<\cdots\}$ be a partition of $[0,\infty)$
and $\tau_{n}:=t_{\boldsymbol{\tau}}^{n}-t_{\boldsymbol{\tau}}^{n-1}$. Consider $T>0$, $N\in\mathbb{N}$ such that $T\in (t_{\boldsymbol{\tau}}^{N-1},t_{\boldsymbol{\tau}}^{N}]$, and the following functions defined on the interval $[0,T]$:
\begin{align}
\mathcal{T}_{\boldsymbol{\tau}}(t)  &  :=t_{\boldsymbol{\tau}}^{n}, \text{ \ for \ } t\in(t_{\boldsymbol{\tau}}^{n-1},t_{\boldsymbol{\tau}}^{n}], \label{eq:tempo-disc}\\
\widetilde{\lambda}_{\boldsymbol{\tau}}(t)  &  :=\lambda(t_{\boldsymbol{\tau}}^{n}), \text{ \ for \ } t\in(t_{\boldsymbol{\tau}}^{n-1},t_{\boldsymbol{\tau}}^{n}], \label{eq:lambda-discret}\\
l_{\boldsymbol{\tau}}(t)  &  :=\frac{t-t_{\boldsymbol{\tau}}^{n-1}}{\tau_{n}}, \text{ \ for \ } t\in(t_{\boldsymbol{\tau}}^{n-1},t_{\boldsymbol{\tau}}^{n}], \label{eq:Ltau}\\
d_{\boldsymbol{\tau}}^{2}(t;V)  &  :=(1-l_{\boldsymbol{\tau}}(t))d^{2}(U_{\boldsymbol{\tau}}^{n-1},V)+l_{\boldsymbol{\tau}}(t)d^{2}%
(U_{\boldsymbol{\tau}}^{n},V), \text{ \ for \ } t\in(t_{\boldsymbol{\tau}}^{n-1},t_{\boldsymbol{\tau}}^{n}], \label{eq:dist-interpolat}\\
\mathcal{E}_{\boldsymbol{\tau}}(t)  &  :=(1-l_{\boldsymbol{\tau}}(t))\mathcal{E}(t_{\boldsymbol{\tau}}^{n-1},U_{\boldsymbol{\tau}}^{n-1})+l_{\boldsymbol{\tau}}(t)\mathcal{E}(t_{\boldsymbol{\tau}}^{n},U_{t_{\boldsymbol{\tau}}}^{n}), \text{ \ for \ } t\in(t_{\boldsymbol{\tau}}^{n-1},t_{\boldsymbol{\tau}}^{n}], \label{eq:functional-interpolation}\\
\underline{U}_{\boldsymbol{\tau}}(t)  &  :=U_{\boldsymbol{\tau}}^{n-1},\ \overline{U}_{\boldsymbol{\tau}}(t):=U_{\boldsymbol{\tau}}^{n}, \text{ \ for \ } t\in(t_{\boldsymbol{\tau}}^{n-1},t_{\boldsymbol{\tau}}^{n}]. \label{eq:sol-interpolante1}%
\end{align}
Also, we consider $\overline{U}_{\boldsymbol {\tau}}(0)=\underline{U}%
_{\boldsymbol{\tau}}(0):=U_{\boldsymbol{\tau}}^{0}$. The functions in
(\ref{eq:sol-interpolante1}) are called approximate solutions for
(\ref{eq:sist-grad-flow}) corresponding to the data $U_{\boldsymbol{\tau}}%
^{0}$.

Taking $u_{\tau}^{t}=U_{\boldsymbol{\tau}}^{n}$, $u=U_{\boldsymbol{\tau}}%
^{n-1}$ and $v=V,$ we can rewrite (\ref{eq:formulation-convex-discret}) as
\begin{align}
\frac{1}{2}\frac{d}{dt}d_{\boldsymbol{\tau}}^{2}(t;V)+\frac{\widetilde
{\lambda}_{\boldsymbol{\tau}}(t)}{2}d^{2}(\overline{U}_{\boldsymbol{\tau}}%
(t),V)+\mathcal{E}_{\boldsymbol{\tau}}(t)-\mathcal{E}(\mathcal{T}%
_{\boldsymbol{\tau}}(t),V)  &  \leq\label{eq:desig-var-interp}\\
\frac{1}{2}\mathscr{R}_{\boldsymbol{\tau}}(t)+(1-l_{\boldsymbol{\tau}}%
(t))(\mathcal{E}(t_{\boldsymbol{\tau}}^{n-1},U_{\boldsymbol{\tau}}%
^{n-1})-\mathcal{E}(t_{\boldsymbol{\tau}}^{n},U_{\boldsymbol{\tau}}^{n-1}))
&  ,\nonumber
\end{align}
for $t\in(t_{\boldsymbol{\tau}}^{n-1},t_{\boldsymbol{\tau}}^{n}]$, where
\begin{equation}
\frac{1}{2}\mathscr{R}_{\boldsymbol{\tau}}(t):=(1-l_{\boldsymbol{\tau}}%
(t))(\mathcal{E}(t_{\boldsymbol{\tau}}^{n},U_{\boldsymbol{\tau}}%
^{n-1})-\mathcal{E}(t_{\boldsymbol{\tau}}^{n},U_{\boldsymbol{\tau}}%
^{n}))-\frac{1}{2\tau_{n}}d^{2}(U_{\boldsymbol{\tau}}^{n-1}%
,U_{\boldsymbol{\tau}}^{n}). \label{eq-residuo-R}%
\end{equation}

With this notation, we have the next estimate.

\begin{lem}
\label{lem:desig-var-interp} Assume \textbf{E1} to \textbf{E5}. For a
partition $\boldsymbol{\tau}$ with $\vert\boldsymbol{\tau}\vert<\tau^{*}$,
define the residual term
\begin{equation}
\mathscr{D}_{\boldsymbol{\tau}}(t):=(1-l_{\boldsymbol{\tau}}(t))d(\overline
{U}_{\boldsymbol{\tau}}(t),\underline{U}_{\boldsymbol{\tau}}(t)).
\label{eq:residuo-D}%
\end{equation}
We have that
\begin{align}
&  \frac{1}{2}\frac{d}{dt}d_{\boldsymbol{\tau}}^{2}(t;V)+\frac{\widetilde
{\lambda}_{\boldsymbol{\tau}}(t)}{2}d_{\boldsymbol{\tau}}^{2}(t;V)-\left(
\widetilde{\lambda}_{\boldsymbol{\tau}}^{+}(t)d(\overline{U}%
_{\boldsymbol{\tau}}(t),\underline{U}_{\boldsymbol{\tau}}(t))+\widetilde
{\lambda}_{\boldsymbol{\tau}}^{-}(t)\mathscr{D}_{\boldsymbol{\tau}}(t)\right)
d_{\boldsymbol{\tau}}(t;V)\nonumber\label{eq:desig-var-interp1}\\
&  +\mathcal{E}_{\boldsymbol{\tau}}(t)-\mathcal{E}(\mathcal{T}%
_{\boldsymbol{\tau}}(t),V)\leq\frac{1}{2}\mathscr{R}_{\boldsymbol{\tau}}%
+\frac{\widetilde{\lambda}_{\boldsymbol{\tau}}^{-}(t)}{2}%
\mathscr{D}_{\boldsymbol{\tau}}^{2}+(1-l_{\boldsymbol{\tau}})(\mathcal{E}%
(t_{\boldsymbol{\tau}}^{n-1},\underline{U}_{\boldsymbol{\tau}})-\mathcal{E}%
(t_{\boldsymbol{\tau}}^{n},\underline{U}_{\boldsymbol{\tau}})),
\end{align}
for all $V\in\mathbf{D}$ and almost every point $t\in[0,T]$.
\end{lem}

\begin{proof}
A detailed proof for the case $\tilde{\lambda}_{\boldsymbol{\tau}}(t)<0$ can be found in \cite[pg.88]{Ambrosio}. Let us explicit the proof for $\tilde{\lambda}_{\boldsymbol{\tau}}(t)>0$. For that, we can suppose that
$$d(U^n_{\boldsymbol{\tau}},V)<d(U^{n-1}_{\boldsymbol{\tau}},V),$$
and estimate
$$d^2(\overline{U}_{\boldsymbol{\tau}}(t),V)-d^2_{\boldsymbol{\tau}}(t;V)=(1-l_{\boldsymbol{\tau}})
\left(d^2(\overline{U}_{\boldsymbol{\tau}}(t),V)-d^2(\underline{U}_{\boldsymbol{\tau}}(t),V)\right).$$
Thus, we have
\begin{eqnarray*}
d^2(\overline{U}_{\boldsymbol{\tau}}(t),V)-d^2_{\boldsymbol{\tau}}(t;V)&\geq&-d(U^n_{\boldsymbol{\tau}},
U^{n-1}_{\boldsymbol{\tau}})(d(U^n_{\boldsymbol{\tau}},V)+d(U^{n-1}_{\boldsymbol{\tau}},V))\\
& &+l_{\boldsymbol{\tau}}d(U^n_{\boldsymbol{\tau}},U^{n-1}_{\boldsymbol{\tau}})(d(U^{n-1}_{\boldsymbol{\tau}},V)
-d(U^{n}_{\boldsymbol{\tau}},V))\\
&=&-d(U^n_{\boldsymbol{\tau}},U^{n-1}_{\boldsymbol{\tau}})\left((1-l_{\boldsymbol{\tau}})d(U^{n-1}_{\boldsymbol{\tau}}
,V)\right.\\
& &\left.+l_{\boldsymbol{\tau}}d(U^{n}_{\boldsymbol{\tau}},V)+d(U^{n}_{\boldsymbol{\tau}},V)\right)\\
&\geq&-2d(U^n_{\boldsymbol{\tau}},U^{n-1}_{\boldsymbol{\tau}})\left((1-l_{\boldsymbol{\tau}})d(U^{n-1}_{\boldsymbol{\tau}}
,V)+l_{\boldsymbol{\tau}}d(U^{n}_{\boldsymbol{\tau}},V)\right)\\
&\geq&-2d(U^n_{\boldsymbol{\tau}},U^{n-1}_{\boldsymbol{\tau}})d_{\boldsymbol{\tau}}(t;V),
\end{eqnarray*}
which together with \eqref{eq:desig-var-interp} gives the desired result.
\end{proof}The next result is a slightly modified version of the Gronwall
Lemma in \cite[Lemma 4.1.8]{Ambrosio}. The proof is the same and we omit it.

\begin{lem}
\label{lem:Gronwall-lemma} Let $x:[0,\infty)\rightarrow\mathbb{R}$ be a
locally absolutely continuous function and let $a,b,\tilde{\lambda}\in L_{loc}%
^{1}([0,\infty))$ be such that
\begin{equation}
\frac{d}{dt}x^{2}(t)+2\widetilde{\lambda}(t)x^{2}(t)\leq a(t)+2b(t)x(t)\text{
a.e. }t\geq0. \label{eq:gronwall-version}%
\end{equation}
For $T>0,$ we have that
\[
e^{\alpha(T)}|x(T)|\leq\sqrt{\left(  x^{2}(0)+\sup_{t\in\lbrack0,T]}\int
_{0}^{t}{e^{2\alpha(s)}a(s)\ ds}\right)  ^{+}}+2\int_{0}^{T}{e^{\alpha
(t)}|b(t)|\ dt},
\]
where $\alpha(t)=\int_{0}^{t}{\widetilde{\lambda}(s)\ ds}$.
\end{lem}

\subsection{\bigskip More two interpolation terms}

\hspace{0.5cm}In this subsection we consider two interpolation functions that
depend on two partitions $\boldsymbol{\tau}$ and $\boldsymbol{\eta}$ of
$[0,\infty)$ with $|\boldsymbol{\tau}|,|\boldsymbol{\eta}|<\tau^{\ast}$. So
far, we have define two residual terms $\mathscr{R}_{\boldsymbol{\tau}}$ and
$\mathscr{D}_{\boldsymbol{\tau}}$ in (\ref{eq-residuo-R}) and
(\ref{eq:residuo-D}), respectively. Another one that we will work with is

\begin{align}
G_{\boldsymbol{\tau\eta}}(t)  &  :=2(1-l_{\boldsymbol{\tau}}(t))\left[
\mathcal{E}(\mathcal{T}_{\boldsymbol{\eta}}(t),\underline{U}%
_{\boldsymbol{\tau}}(t))-\mathcal{E}(\mathcal{T}_{\boldsymbol{\tau}}%
(t),\underline{U}_{\boldsymbol {\tau}}(t))\right] \nonumber\\
&  \ +2l_{\boldsymbol{\tau}}(t)\left[  \mathcal{E}(\mathcal{T}%
_{\boldsymbol{\eta}}(t),\overline{U}_{\boldsymbol{\tau}}(t))-\mathcal{E}%
(\mathcal{T}_{\boldsymbol{\tau}}(t),\overline{U}_{\boldsymbol{\tau}}(t))\right], \text{ \ for \ } t\in\lbrack0,T]. \label{eq:residuo-G}
\end{align}
Define also the interpolation function
$d_{\boldsymbol{\tau}\boldsymbol{\eta}}^{2}(t,s)$ as
\[
d_{\boldsymbol{\tau}\boldsymbol{\eta}}^{2}(t,s)=(1-l_{\boldsymbol{\eta}}%
(s))d_{\boldsymbol{\tau}}^{2}(t,\underline{U}_{\boldsymbol{\eta}}%
(s))+l_{\boldsymbol{\eta}}(s)d_{\boldsymbol{\tau}}^{2}(t,\overline
{U}_{\boldsymbol{\eta}}(s)).
\]
Taking in (\ref{eq:desig-var-interp1}) a convex combination, with coefficients
$(1-l_{\boldsymbol{\eta}}(t))$ and $l_{\boldsymbol{\eta}}(t)$ for
$V=\underline{U}_{\boldsymbol{\eta}}(t)$ and $V=\overline{U}%
_{\boldsymbol{\eta}}(t)$ respectively, we arrive at
\begin{align*}
\frac{d}{dt}d_{\boldsymbol{\tau\eta}}^{2}(t,t)+(\widetilde{\lambda
}_{\boldsymbol{\tau}}+\widetilde{\lambda}_{\boldsymbol {\eta}}%
)d_{\boldsymbol{\tau\eta}}^{2}(t,t)  &  \leq2\left[  \widetilde{\lambda
}_{\boldsymbol{\tau}}^{+}d(\overline{U}_{\boldsymbol{\tau}}(t),\underline
{U}_{\boldsymbol{\tau}}(t))+\widetilde{\lambda}_{\boldsymbol{\eta}}%
^{+}d(\overline{U}_{\boldsymbol{\eta}}(t),\underline{U}_{\boldsymbol{\eta}}%
(t))\right. \\
&  \left.  +\widetilde{\lambda}_{\boldsymbol{\tau}}^{-}%
(t)\mathscr{D}_{\boldsymbol{\tau}}(t)+\widetilde{\lambda}_{\boldsymbol{\eta}}%
^{-}(t)\mathscr{D}_{\boldsymbol{\eta}}(t)\right]  d_{\boldsymbol{\tau\eta}}%
(t,t)+\mathscr{R}_{\boldsymbol{\tau}}(t)\\
&  +\mathscr{R}_{\boldsymbol{\eta}}(t)+\widetilde{\lambda}_{\boldsymbol{\tau}}%
^{-}(t)\mathscr{D}_{\boldsymbol{\tau}}^{2}(t)+\widetilde{\lambda
}_{\boldsymbol{\eta}}^{-}(t)\mathscr{D}_{\boldsymbol{\eta}}^{2}%
(t)+G_{\boldsymbol{\tau\eta}}(t)\\
&  +G_{\boldsymbol{\eta\tau}}(t).
\end{align*}
Now we can use Lemma \ref{lem:Gronwall-lemma} in the last inequality in order
to estimate
\begin{align}
e^{\alpha_{\boldsymbol{\tau\eta}}(t)}d_{\boldsymbol{\tau\eta}}(t,t)  &
\leq\left(  d^{2}(U_{\boldsymbol{\tau}}^{0},U_{\boldsymbol{\eta}}^{0}%
)+\int_{0}^{t}{e^{2\alpha_{\boldsymbol{\tau\eta}}(s)}\sum
_{\boldsymbol{\theta}\in\{\boldsymbol {\tau,\eta}\}}{\left(
\mathscr{R}_{\boldsymbol{\theta}}^{+}(s)+\widetilde{\lambda}%
_{\boldsymbol{\theta}}^{-}(s)\mathscr{D}_{\boldsymbol{\theta}}^{2}(s)\right)
\ ds}}\right. \nonumber\\
&  \left.  +\int_{0}^{t}{e^{2\alpha_{\boldsymbol{\tau\eta}}(s)}\left(
G_{\boldsymbol{\eta\tau}}^{+}(s)+G_{\boldsymbol {\tau\eta}}^{+}(s)\right)
\ ds}\right)  ^{1/2}\nonumber\\
&  \ +\int_{0}^{t}{e^{\alpha_{\boldsymbol{\tau\eta}}(s)}\left(  \widetilde
{\lambda}_{\boldsymbol{\tau}}^{+}(s)d(\overline{U}_{\boldsymbol{\tau}}%
(s),\underline{U}_{\boldsymbol{\tau}}(s))+\widetilde{\lambda}%
_{\boldsymbol{\eta}}^{+}(s)d(\overline{U}_{\boldsymbol{\eta}}(s),\underline
{U}_{\boldsymbol{\eta}}(s))\right.  }\nonumber\\
&  \left.  +\widetilde{\lambda}_{\boldsymbol{\tau}}^{-}%
(s)\mathscr{D}_{\boldsymbol{\tau}}(s)+\widetilde{\lambda}_{\boldsymbol{\eta}}%
^{-}(s)\mathscr{D}_{\boldsymbol{\eta}}(s)\right)  \ ds,
\label{eq:estim-before-residuo}%
\end{align}
for all $t\geq0$, where $\alpha_{\boldsymbol{\tau\eta}}(t):=\int_{0}%
^{t}{\widetilde{\lambda}_{\boldsymbol{\tau}}(s)+\widetilde{\lambda
}_{\boldsymbol{\eta}}(s)\ ds}$.

\subsection{Convergence of the approximate solutions}

\label{subsect:convergence}\hspace{0.5cm}In this section, we deal with the
convergence of the approximate solutions $\overline{U}_{\boldsymbol{\tau}}$
and $\underline{U}_{\boldsymbol{\tau}}$. Using the minimizer property of
$U_{\boldsymbol{\tau}}^{j}$ and direct calculations, one can obtain
\begin{equation}
\int_{0}^{t}{e^{2\alpha_{\boldsymbol{\tau\eta}}(s)}%
(\mathscr{R}_{\boldsymbol{\tau}}^{+}(s)+\widetilde{\lambda}%
_{\boldsymbol{\tau}}^{-}(s)\mathscr{D}_{\boldsymbol{\tau}}^{2}(s))\ ds}\leq
C|\boldsymbol{\tau}|\sum_{j=1}^{n}(\mathcal{E}(t_{\boldsymbol{\tau}}%
^{j},U_{\boldsymbol{\tau}}^{j-1})-\mathcal{E}(t_{\boldsymbol{\tau}}%
^{j},U_{\boldsymbol{\tau}}^{j})) \label{eq:estimativa-residuos-R-D}%
\end{equation}%
\begin{equation}
\left(  \int_{0}^{t}{e^{\alpha_{\boldsymbol{\tau\eta}}(s)}\widetilde{\lambda
}_{\boldsymbol{\tau}}^{-}(s)\mathscr{D}_{\boldsymbol{\tau}}(s)\ ds}\right)
^{2}\leq C|\boldsymbol{\tau}|^{2}\sum_{j=1}^{n}(\mathcal{E}%
(t_{\boldsymbol{\tau}}^{j},U_{\boldsymbol{\tau}}^{j-1})-\mathcal{E}%
(t_{\boldsymbol{\tau}}^{j},U_{\boldsymbol{\tau}}^{j})) \label{eq:estimativa-D}%
\end{equation}%
\begin{equation}
\left(  \int_{0}^{t}{e^{\alpha_{\boldsymbol{\tau\eta}}(s)}\widetilde{\lambda
}_{\boldsymbol{\tau}}^{+}(s)d(\overline{U}_{\boldsymbol{\tau}}(s),\underline
{U}_{\boldsymbol{\tau}}(s))\ ds}\right)  ^{2}\leq C(T)|\boldsymbol{\tau}|^{2}%
\sum_{j=1}^{n}(\mathcal{E}(t_{\boldsymbol{\tau}}^{j},U_{\boldsymbol{\tau}}%
^{j-1})-\mathcal{E}(t_{\boldsymbol{\tau}}^{j},U_{\boldsymbol{\tau}}^{j})),
\label{eq:estimativa-d}%
\end{equation}
for $T\in(t_{\boldsymbol{\tau}}^{N-1},t_{\boldsymbol{\tau}}^{N}],$ $0\leq
t\leq T$, and $n\leq N.$

The above estimates give some control on the residual terms
$\mathscr{R}_{\boldsymbol{\tau}}$ and $\mathscr{D}_{\boldsymbol{\tau}}$. Next,
we provide an explicit estimate for the residual term (\ref{eq:residuo-G}).
This could be useful to obtain convergence rates of approximate solutions to
the gradient flow solutions. Recall the standard notations $a\wedge
b=min\{a,b\}$ and $a\vee b=max\{a,b\}$.

\begin{prop}
\label{prop:residuo-G} Assume \textbf{E1} to \textbf{E5} and the boundedness
condition (\ref{eq:limit-sec-discret}). Let
$\boldsymbol{\tau},\boldsymbol{\eta}$ be two partitions of $[0,+\infty)$ with
$\vert\boldsymbol{\tau}\vert$, $\vert\boldsymbol{\eta}\vert$ small enough as
in Lemma \ref{lemma:limit-sec-discret}. For $T>0,$ choose $N,K\in\mathbb{N}$,
such that $T\in(t_{\boldsymbol{\tau}}^{N-1},t_{\boldsymbol{\tau}}^{N}%
]\cap(t_{\boldsymbol{\tau}}^{K-1},t_{\boldsymbol{\tau}}^{K}]$. There is a
constant $C=C(T,S,\tau^{\ast},\mathcal{E})>0$ such that
\begin{equation}
\int_{0}^{T}{G_{\boldsymbol{\tau\eta}}^{+}(t)\ dt}\leq
C(|\boldsymbol{\tau}|+|\boldsymbol{\eta}|). \label{eq:proposition-est-G}%
\end{equation}

\end{prop}

\begin{proof}
Denote $I_{\boldsymbol{\tau}}^{n}=(t_{\boldsymbol{\tau}}^{n-1},t_{\boldsymbol{\tau}}^{n}]$. For
$t_{\boldsymbol{\tau}}^{1}$, let $k_1$ be the greatest integer satisfying $t_{\boldsymbol{\eta}}^{k_1-1}<
t_{\boldsymbol{\tau}}^{1}$. If $t_{\boldsymbol{\eta}}^{k_1}=t_{\boldsymbol{\tau}}^{1}$ define $J_1=I_{\boldsymbol{\tau}}
^{1}$. Otherwise, choose $n_1\leq N$ as the greatest integer with the property $t_{\boldsymbol{\tau}}^{n_1}<t_{\boldsymbol
{\eta}}^{k_1}$ and define $J_1=I_{\boldsymbol{\tau}}^1\cup\cdots\cup I_{\boldsymbol{\tau}}^{n_1}$. In both cases, we have $\mathcal{T}_{\boldsymbol{\eta}}(t_{\boldsymbol{\tau}}^{n_1})=t_{\boldsymbol{\eta}}^{k_1}$ and then
\begin{eqnarray*}
\int_{J^1_{\boldsymbol{\tau}}}{(1-l_{\boldsymbol{\tau}}(t)){\int_{\mathcal{T}_{\boldsymbol{\tau}}(t)\wedge\mathcal{T
}_{\boldsymbol{\eta}}(t)}^{\mathcal{T}_{\boldsymbol{\tau}}(t)\vee\mathcal{T}_{\boldsymbol{\eta}}(t)}}{\beta(s)\ ds}
\ dt}&\leq&\int_{J^1_{\boldsymbol{\tau}}}{(1-l_{\boldsymbol{\tau}}(t))\ dt}\int_{0}^{\mathcal{T}_{\boldsymbol{\tau}}
(t^{n_1}_{\boldsymbol{\tau}})\vee\mathcal{T}_{\boldsymbol{\eta}}(t^{n_1}_{\boldsymbol{\tau}})}{\beta(s)\ ds}\\
&\leq&(\vert\boldsymbol{\tau}\vert+\vert\boldsymbol{\eta}\vert)\int_{0}^{t^{k_1}_{\boldsymbol{\eta}}}{\beta(s)\ ds}.
\end{eqnarray*}
If $t^{n_1+1}_{\boldsymbol{\tau}}=t^{k_1}_{\boldsymbol{\eta}}$, define $J^2_{\boldsymbol{\tau}}
=I^{n_1+1}_{\boldsymbol{\tau}}$. Otherwise, take the greatest integer $k_2\in\mathbb{N}$ such that $t^{k_2-1}_{\boldsymbol{\eta}}
< t^{n_1+1}_{\boldsymbol{\tau}}$. In the case $t^{k_2}_{\boldsymbol{\eta}}= t^{n_1+1}_{\boldsymbol{\tau}}$, define $J^2_{\boldsymbol{\tau}}=I^{n_1+1}_{\boldsymbol{\tau}}$. Otherwise, take the greatest integer $n_2\leq N$ such that $t^{n_2}_{\boldsymbol{\tau}}< t^{k_2}_{\boldsymbol{\eta}}$ and define $J^2_{\boldsymbol{\tau}}=
I^{n_1+1}_{\boldsymbol{\tau}}\cup\cdots\cup I^{n_2}_{\boldsymbol{\tau}}$. Noting that $\mathcal{T}_{\boldsymbol{\eta}}
(t^{n_2}_{\boldsymbol{\tau}})=t^{k_2}_{\boldsymbol{\eta}}$ and $\mathcal{T}_{\boldsymbol{\tau}}(t)\geq
t^{n_1+1}_{\boldsymbol{\tau}}$, we get
\begin{eqnarray*}
\int_{J^2_{\boldsymbol{\tau}}}{(1-l_{\boldsymbol{\tau}}(t)){\int_{\mathcal{T}_{\boldsymbol{\tau}}(t)\wedge\mathcal{T
}_{\boldsymbol{\eta}}(t)}^{\mathcal{T}_{\boldsymbol{\tau}}(t)\vee\mathcal{T}_{\boldsymbol{\eta}}(t)}}{\beta(s)\ ds}
\ dt}&\leq&\int_{J^2_{\boldsymbol{\tau}}}{(1-l_{\boldsymbol{\tau}}(t))\ dt}\int_{t^{n_1+1}_{\boldsymbol{\tau}}\wedge\mathcal{T}_{\boldsymbol{\eta}}(t^{n_1}_{\boldsymbol{\tau}})}^{\mathcal{T}_
{\boldsymbol{\tau}}(t^{n_2}_{\boldsymbol{\tau}})\vee\mathcal{T}_{\boldsymbol{\eta}}(t^{n_2}_{\boldsymbol{\tau}})}{
\beta(s)\ ds}\\
&\leq&(\vert\boldsymbol{\tau}\vert+\vert\boldsymbol{\eta}\vert)\int_{t^{k_1}_{\boldsymbol{\eta}}}^{t^{k_2}_
{\boldsymbol{\eta}}}{\beta(s)\ ds}.
\end{eqnarray*}
Proceeding inductively, and adding estimates obtained in the process, we arrive at
\begin{equation}
\label{eq:proposicao-est-G1}
\int_{0}^T{(1-l_{\boldsymbol{\tau}}(t)){\int_{\mathcal{T}_{\boldsymbol{\tau}}(t)\wedge\mathcal{T
}_{\boldsymbol{\eta}}(t)}^{\mathcal{T}_{\boldsymbol{\tau}}(t)\vee\mathcal{T}_{\boldsymbol{\eta}}(t)}}{\beta(s)\ ds}
\ dt}\leq(\vert\boldsymbol{\tau}\vert+\vert\boldsymbol{\eta}\vert)\int_{0}^{T+\tau^*}{\beta(s)\ ds}.
\end{equation}
Analogously,
\begin{equation}
\label{eq:proposicao-est-G2}
\int_{0}^T{l_{\boldsymbol{\tau}}(t){\int_{\mathcal{T}_{\boldsymbol{\tau}}(t)\wedge\mathcal{T
}_{\boldsymbol{\eta}}(t)}^{\mathcal{T}_{\boldsymbol{\tau}}(t)\vee\mathcal{T}_{\boldsymbol{\eta}}(t)}}{\beta(s)\ ds}
\ dt}\leq(\vert\boldsymbol{\tau}\vert+\vert\boldsymbol{\eta}\vert)\int_{0}^{T+\tau^*}{\beta(s)\ ds}.
\end{equation}
Adding (\ref{eq:proposicao-est-G1}) and (\ref{eq:proposicao-est-G2}), we get
\begin{equation}
\label{eq:proposicao-est-G3}
\int_{0}^T{{\int_{\mathcal{T}_{\boldsymbol{\tau}}(t)\wedge\mathcal{T}_{\boldsymbol{\eta}}(t)}^{\mathcal{T}_{\boldsymbol{\tau}}(t)\vee\mathcal{T}_{\boldsymbol{\eta}}(t)}}{\beta(s)\ ds}
\ dt}\leq 2(\vert\boldsymbol{\tau}\vert+\vert\boldsymbol{\eta}\vert)\int_{0}^{T+\tau^*}{\beta(s)\ ds}.
\end{equation}
Now, recalling (\ref{eq:residuo-D}) and the property \textbf{E5}, and using the first estimate in (\ref{eq:limit-sec-discret1}), for $t\in[0,T]$ it follows that
\begin{eqnarray}
G^+_{\boldsymbol{\tau}\boldsymbol{\eta}}(t)&\leq&2(1-l_{\boldsymbol{\tau}}(t))\int_{\mathcal{T}_{\boldsymbol{\tau}}
(t)\wedge\mathcal{T}_{\boldsymbol{\eta}}(t)}^{\mathcal{T}_{\boldsymbol{\tau}}(t)\vee\mathcal{T}_{\boldsymbol{\eta}}
(t)}{\beta(s)\ ds}(1+d^2(u^*,\overline{U}_{\boldsymbol{\tau}})(t))\nonumber\\
&\ &+2l_{\boldsymbol{\tau}}(t)\int_{\mathcal{T}_{\boldsymbol{\tau}}
(t)\wedge\mathcal{T}_{\boldsymbol{\eta}}(t)}^{\mathcal{T}_{\boldsymbol{\tau}}(t)\vee\mathcal{T}_{\boldsymbol{\eta}}
(t)}{\beta(s)\ ds}(1+d^2(u^*,\underline{U}_{\boldsymbol{\tau}})(t))\nonumber\\
&\leq&C(S,T,\tau^*,\mathcal{E})\int_{\mathcal{T}_{\boldsymbol{\tau}}(t)\wedge\mathcal{T}_{\boldsymbol{\eta}}(t)}^{\mathcal{T}_{\boldsymbol{\tau}}(t)\vee\mathcal{T}
_{\boldsymbol{\eta}}(t)}{\beta(s)\ ds}.\label{aux-est-3}
\end{eqnarray}
Finally, we conclude by integrating \eqref{aux-est-3} over $[0,T]$ and using (\ref{eq:proposicao-est-G3}).
\end{proof}

In the present section and in Section \ref{sect:construc-and-propert}, we have
obtained some properties and estimates for $\mathcal{E}(t,u)$ and the implicit
variational scheme (\ref{eq:minimization-problem}) associated to the problem
(\ref{eq:sist-grad-flow})-(\ref{eq:sist-grad-flow-initial-data}). After doing
that, we are in position for proceeding as in \cite[pag. 91-92]{Ambrosio} and showing that the approximate solutions (\ref{eq:sol-interpolante1}) converge uniformly in $[0,T]$ as $|\boldsymbol{\tau}|\rightarrow0.$

\begin{teor}
\label{teor:convergence-scheme-discret} Assume \textbf{E1} to \textbf{E5} and
the condition
\begin{equation}
\lim_{|\boldsymbol{\tau}|\rightarrow0}d(U_{\boldsymbol{\tau}}^{0}%
,u_{0})=0,\hspace{1cm}\sup_{\boldsymbol{\tau}}\mathcal{E}%
(0,U_{\boldsymbol{\tau}}^{0})=S<\infty, \label{eq:exist-solucao}%
\end{equation}
for some $u_{0}\in\mathbf{D}$. Then, the approximate solutions $\overline
{U}_{\boldsymbol{\tau}}$ and $\underline{U}_{\boldsymbol{\tau}}$ converge
locally uniformly to a function $u:[0,\infty)\rightarrow X$ satisfying
$u(0)=u_{0}$. Moreover, $u$ is independent of the family
$U_{\boldsymbol{\tau}}^{0}$.
\end{teor}

\begin{rem}
\label{remark:convergence-in-Dbar} In fact, the convergence of the approximate
solutions is valid for $u_{0}\in\bar{\mathbf{D}}.$
\end{rem}

\bigskip
\begin{proof}[Proof of Theorem \ref{teor:convergence-scheme-discret}.] The proof follows essentially the same arguments in \cite{Ambrosio} by taking care of the time-dependence. We give some
steps for the reader convenience. By taking a suitable convex combination, we arrive at
\begin{equation*}
d^2(\overline{U}_{\boldsymbol{\tau}}(t),\overline{U}_{\boldsymbol{\eta}}(t))\leq 3d^2_{\boldsymbol{\tau},\boldsymbol{\eta}}
(t,t)+3C(|\boldsymbol{\tau}|+|\boldsymbol{\eta}|).
\end{equation*}
So, using \eqref{eq:estimativa-residuos-R-D}, \eqref{eq:estimativa-D}, \eqref{eq:estimativa-d} joint with Lemma
\ref{lemma:limit-sec-discret}, and the estimate \eqref{eq:estim-before-residuo}, we obtain
\begin{eqnarray*}
d_{\boldsymbol{\tau},\boldsymbol{\eta}}(t,t)&\leq& \left(d^2(U^0_{\boldsymbol{\tau}},U_{\boldsymbol{\eta}})+C(|
\boldsymbol{\tau}|+|\boldsymbol{\eta}|)+\int_0^t{e^{2\alpha_{\boldsymbol{\tau},\boldsymbol{\eta}}(t)}(G^+_{
\boldsymbol{\tau},\boldsymbol{\eta}}(t)+G^+_{
\boldsymbol{\eta},\boldsymbol{\tau}}(t))\ dt}\right)^{1/2}\\
& &+C(|\boldsymbol{\tau}|+|\boldsymbol{\eta}|).
\end{eqnarray*}
We conclude the convergence by using Proposition \ref{prop:residuo-G} and the completeness of the space $X$.
\end{proof}

\section{Regularity}

\label{sect:regularity}\hspace{0.5cm}In this section we show that the function
$u$ obtained in Theorem \ref{teor:convergence-scheme-discret} is, in fact, a
solution for (\ref{eq:sist-grad-flow})-(\ref{eq:sist-grad-flow-initial-data})
in the sense of Definition \ref{defn:def-solution}. For that matter, we need
to show some regularity properties for $u$. We begin by recalling the De
Giorgi interpolation.

\begin{defn}
\label{definition:DiGiorgi} Let $(U_{\boldsymbol{\tau}}^{n})_{n}$ be a
solution for the variational scheme (\ref{eq:minimization-problem}), defined for $t_{\boldsymbol{\tau}}^{n}\leq T+\tau^{\ast}$. Define the De Giorgi interpolation
\[
\widetilde{U}_{\boldsymbol{\tau}}(t)=\widetilde{U}_{\boldsymbol{\tau}}%
(t_{\boldsymbol{\tau}}^{n-1}+\delta), \text{ \ for \ } t\in(t_{\boldsymbol{\tau}}%
^{n-1},t_{\boldsymbol{\tau}}^{n}] \text{ \ and }\delta=t-t_{\boldsymbol{\tau}}%
^{n-1},
\]
as the unique minimizer of the functional $v\in X\rightarrow
\mathbf{E}(t_{\boldsymbol{\tau}}^{n-1}+\delta,\delta,U_{\boldsymbol{\tau}}%
^{n-1},v).$
\end{defn}

We have that the De Giorgi interpolation also converges locally uniformly to
the same function $u$ in Theorem \ref{teor:convergence-scheme-discret}.

\begin{prop}
\label{prop:DiGiorgi-interp-converg} Assume the same hypotheses of Theorem
\ref{teor:convergence-scheme-discret}. There is a constant $C>0$ independent
of $\boldsymbol{\tau}$ such that
\begin{equation}
d^{2}(\underline{U}_{\boldsymbol{\tau}}(t),\widetilde{U}_{\boldsymbol{\tau}}%
(t))\leq|\boldsymbol{\tau}|C\left(  1+\frac{t-t_{\boldsymbol{\tau}}^{n-1}%
}{\mathcal{T}_{\boldsymbol{\tau}}(t)-t}\int_{t}^{\mathcal{T}%
_{\boldsymbol{\tau}}(t)}{\beta(r)\ dr}\right)  ,
\label{eq:converg-interp-degiorgi}%
\end{equation}
for all $t\in(t_{\boldsymbol{\tau}}^{n-1},t_{\boldsymbol{\tau}}^{n}]$. Thus,
$\widetilde{U}_{\boldsymbol{\tau}}$ converges to the function $u$ given in Theorem
\ref{teor:convergence-scheme-discret} a.e. in $\lbrack0,T]$. Moreover, the convergence is uniform
provided that the function $\beta$ in {\bf{E3}} belongs to $L_{loc}^{\infty
}([0,\infty))$.
\end{prop}

\begin{proof}
Let $N\in\mathbb{N}$ be such that $T\in(t_{\boldsymbol{\tau}}^{N-1},t_{\boldsymbol{\tau}}^N]$. First, we will show that the discrete solution $(U_{\boldsymbol{\tau}}^n)_{n=0}^{N}$ satisfies the inequality
\begin{equation}
\label{eq:prop-DiGiorgi-conv-aux}
\mathcal{E}(t_{\boldsymbol{\tau}}^n,U_{\boldsymbol{\tau}}^n)\leq \mathcal{E}(0,U_{\boldsymbol{\tau}}^0)+C,
\end{equation}
for some constant $C$ independent of $\boldsymbol{\tau}$. In fact, by using \textbf{E3} and the minimizer property (\ref{eq:minimization-problem}) of $U_{\boldsymbol{\tau}}^n$, we obtain
\begin{eqnarray*}
\mathcal{E}(t_{\boldsymbol{\tau}}^n,U_{\boldsymbol{\tau}}^n)\leq\mathcal{E}(t_{\boldsymbol{\tau}}^{n-1},U_
{\boldsymbol{\tau}}^{n-1})+\int_{t_{\boldsymbol{\tau}}^{n-1}}^{t_{\boldsymbol{\tau}}^{n}}{\beta(r)\ dr}(1+d(u^*,
U_{\boldsymbol{\tau}}^{n-1})).
\end{eqnarray*}
Recall that $d(u^*,U_{\boldsymbol{\tau}}^{n-1})$ is bounded by a constant $C$ that depends on $T$ and is independent of $\boldsymbol{\tau}$. Proceeding inductively, it follows that
\begin{equation*}
\mathcal{E}(t_{\boldsymbol{\tau}}^n,U_{\boldsymbol{\tau}}^n)\leq \mathcal{E}(0,U_{\boldsymbol{\tau}}^0)+C
\int_0^{T+\tau^*}{\beta(r)\ dr},
\end{equation*}
from where we get (\ref{eq:prop-DiGiorgi-conv-aux}). Now, estimate (\ref{eq:estimative-aux2}) in Lemma \ref{lem:minimizer-estimative1} and (\ref{eq:prop-DiGiorgi-conv-aux}) give
\begin{eqnarray*}
d(\widetilde{U}_{\boldsymbol{\tau}}(t),u^*)&\leq& 4\tau^*\left(\mathcal{E}(t^{n-1}_{\boldsymbol{\tau}},
\underline{U}_{\boldsymbol{\tau}}(t))+C(1+d^2(u^*,\underline{U}_{\boldsymbol{\tau}}(t)))\right)\\
&\leq&4\tau^*(\mathcal{E}(0,U_{\boldsymbol{\tau}}^0)+C),
\end{eqnarray*}
for $t\in(t^{n-1}_{\boldsymbol{\tau}},t^{n}_{\boldsymbol{\tau}}]$. Taking $\delta=t-t^{n-1}_{\boldsymbol{\tau}}$ for $t\in(t^{n-1}_{\boldsymbol{\tau}}
,t^{n}_{\boldsymbol{\tau}}]$, we can estimate
\begin{eqnarray*}
\frac{d^2(\underline{U}_{\boldsymbol{\tau}}(t),\widetilde{U}_{\boldsymbol{\tau}}(t))}{2\delta}+\mathcal{E}(t,
\widetilde{U}_{\boldsymbol{\tau}}(t))&\leq&\frac{d^2(\underline{U}_{\boldsymbol{\tau}}(t),\overline{U}_
{\boldsymbol{\tau}}(t))}{2\delta}+\mathcal{E}(t,\overline{U}_{\boldsymbol{\tau}}(t))\\
&\leq&\frac{d^2(\underline{U}_{\boldsymbol{\tau}}(t),\widetilde{U}_{\boldsymbol{\tau}}(t))}{2\tau_n}+\mathcal{E}(t_{\boldsymbol{\tau}}^{n},\widetilde{U}_{\boldsymbol{\tau}}(t))\\
&\ &+\mathcal{E}(t,\overline{U}_{\boldsymbol{\tau}}(t))-\mathcal{E}(t_{\boldsymbol{\tau}}^{n},\overline{U}_{\boldsymbol{\tau}}(t)))\\
&\ &+\left(\frac{1}{2\delta}-\frac{1}{2\tau_n}\right)d^2(\underline{U}_{\boldsymbol{\tau}}(t),\overline{U}_{\boldsymbol{\tau}}(t)).
\end{eqnarray*}
Rearranging terms and using \textbf{E3}, it follows that
\begin{eqnarray*}
\left(\frac{1}{2\delta}-\frac{1}{2\tau_n}\right)d^2(\underline{U}_{\boldsymbol{\tau}}(t),\widetilde{U}_{\boldsymbol
{\tau}}(t))&\leq&\left(\frac{1}{2\delta}-\frac{1}{2\tau_n}\right)d^2(\underline{U}_{\boldsymbol{\tau}}(t),
\overline{U}_{\boldsymbol{\tau}}(t))+\int_t^{t_{\boldsymbol{\tau}}^n}{\beta(r)\ dr}\\
&\ &\times\left(2+d^2(u^*,\overline{U}_{\boldsymbol{\tau}}(t))+d^2(u^*,\widetilde{U}_{\boldsymbol{\tau}}(t))\right).
\end{eqnarray*}
Recalling that $\mathcal{E}(0,U_{\boldsymbol{\tau}}^0)\leq S$ and using \eqref{eq:prop-DiGiorgi-conv-aux},
we obtain (\ref{eq:converg-interp-degiorgi}) and then the convergence of $\widetilde{U}_{\boldsymbol{\tau}}(t)$ to $u(t)$ in the set of Lebesgue points of $\beta$.
\end{proof}

Before proceeding, let us recall a well-known estimate for the slope
$|\partial\mathcal{E}(t)|.$ Recall that $u_{\tau}^{t+\tau}$ stands for the
minimizer of $\mathbf{E}(t+\tau,\tau,u;\cdot)$. Then $u_{\tau}^{t+\tau}\in
Dom(|\partial\mathcal{E}(t+\tau)|)$ and
\begin{equation}
|\partial\mathcal{E}(t+\tau)|(u_{\tau}^{t+\tau})\leq\frac{d(u,u_{\tau}%
^{t+\tau})}{\tau}. \label{eq:est-slope-local}%
\end{equation}
Under the convexity hypothesis \textbf{E5}, we have that the local slope
$|\partial\mathcal{E}(t)|$ is lower semicontinuous and
\begin{equation}
|\partial\mathcal{E}(t)|(u)=\sup_{v\neq u}\left(  \frac{\mathcal{E}%
(t,u)-\mathcal{E}(t,v)}{d(u,v)}+\frac{1}{2}\lambda(t)d(u,v)\right)  ^{+}.
\label{eq:formula-inclinacao}%
\end{equation}

The next lemma will be useful to show $W_{loc}^{1,1}$-regularity for functions
with a certain type of control in their variations.

\begin{lem}
\label{lem:cond-cont-absolute} Let $T>0$ and $f,g,\beta\in L^{1}([0,T])$ be
such that
\[
|f(t)-f(s)|\leq(g(t)+g(s))|t-s|+\int_{s}^{t}{\beta(r)\ dr},
\]
for $s<t$. Then $f\in W^{1,1}([h,T-h]),$ for all $0<h<T/2$.
\end{lem}

\begin{proof}
Since the function $t\to \int_0^t\beta(r)\ dr$ belongs to $W^{1,1}([0,T])$, we have the difference quotient property
\begin{equation}
\label{eq:lem-cond-cont-absoluta} \sup_{0<\vert\tilde{h}\vert<h}\int_h^{T-h}{\left\vert\frac{1}{\tilde{h}}
\int_t^{t+\tilde{h}}{\beta(r)\ dr}\right\vert\ dt}<\infty.
\end{equation}
Using the notation
$$\Delta_{\tilde{h}}(f)(t)=\frac{f(t+\tilde{h})-f(t)}{\tilde{h}},$$
we obtain
\begin{eqnarray*}
\int_h^{T-h}{\vert\Delta_{\tilde{h}}(f)(t)\vert\ dt}&\leq&\int_h^{T-h}{g(t)+g(t+\tilde{h})+\left\vert\frac{1}{\tilde{h}}
\int_t^{t+\tilde{h}}{\beta(r)\ dr}\right\vert\ dt}\\
&\leq&2\left\|g\right\|_{L^1}+\int_h^{T-h}{\left\vert\frac{1}{\tilde{h}}\int_t^{t+\tilde{h}}{\beta(r)\ dr}\right\vert\ dt},
\end{eqnarray*}
which gives the desired regularity by employing a difference quotient argument.
\end{proof}

Now we are ready to show that the limit $u$ in Theorem
\ref{teor:convergence-scheme-discret} is a time-dependent gradient flow in the
sense of Definition \ref{defn:def-solution}.

\begin{teor}
\label{theor:sol-flux-grad} Assume \textbf{E1} to \textbf{E5}. The limit
$u:[0,\infty)\rightarrow X$ in Theorem \ref{teor:convergence-scheme-discret}
is locally absolutely continuous and its metric derivative $|u^{\prime}|$
belongs to $L_{loc}^{2}([0,\infty))$. Moreover, if the function $t\rightarrow
\mathcal{E}(t,u)$ is differentiable for $u\in\mathbf{D}$, its time-derivative
is upper semicontinuous in the $u$-variable (with respect to the metric), and
the property
\begin{equation}
\label{eq:lower-semicon-of-derivative}t_{n}\downarrow t,\ d(u_{n},u)\to0
\text{ as }n\to\infty\Rightarrow\liminf_{n\to\infty}\frac{\mathcal{E}%
(t_{n},u_{n})-\mathcal{E}(t,u_{n})} {t_{n}-t}\geq\partial_{t}\mathcal{E}(t,u)
\end{equation}
holds true, then the function $t\rightarrow\mathcal{E}(t,u(t))$ is absolutely
continuous and satisfies the identity
\begin{equation}
\mathcal{E}(t,u(t))-\mathcal{E}(0,u(0))=\int_{0}^{t}{\partial_{t}%
\mathcal{E}(s,u(s))\ ds}-\frac{1}{2}\int_{0}^{t}{|u^{\prime}|^{2}%
(s)\ ds}-\frac{1}{2}\int_{0}^{t}{|\partial\mathcal{E}(s)|^{2}(u(s))\ ds}.
\label{eq:energy-identity}%
\end{equation}
In particular, $u$ is a solution for (\ref{eq:sist-grad-flow}%
)-(\ref{eq:sist-grad-flow-initial-data}).
\end{teor}

\begin{rem}
\label{remark:solution} Definition \ref{defn:def-solution} does not contain
\eqref{eq:lower-semicon-of-derivative}. Note also that this assumption is used
to prove \eqref{eq:energy-identity} and, in fact, is not necessary to obtain
the absolute continuity of $t\rightarrow\mathcal{E}(t,u(t)).$
\end{rem}

\begin{proof}[Proof of Theorem \ref{theor:sol-flux-grad}.]
Let $T>0$ and denote by
\begin{equation}\label{auxauxdist}
\vert U'_{\boldsymbol{\tau}}\vert(t)=\frac{d(U_{\boldsymbol{\tau}}^{n-1},U_{\boldsymbol{\tau}}^{n})}{\tau_n}
\end{equation}
the discrete derivative of $\overline{U}_{\boldsymbol{\tau}}(t)$ in each interval $(t_{\boldsymbol{\tau}}^{n-1} ,t_{\boldsymbol{\tau}}^{n}]$. By Lemma \ref{lemma:limit-sec-discret}, we have that
\begin{equation}
\int_0^t{\vert U'_{\boldsymbol{\tau}}\vert^2(s)\ ds}\leq C,
\end{equation}
for each $t\in[0,T]$. Thus, we can extract a sequence $\boldsymbol{\tau}_k$ such that $\vert\boldsymbol{\tau}_k\vert\to0$
and $\vert U'_{\boldsymbol{\tau}_k}\vert$ converges weakly in $L^2([0,T])$ for some function $m$. Fix $0\leq s<t\leq T$ and choose $p=p(s)$ and $n=n(s)\in\mathbb{N}$ with $s\in(t_{\boldsymbol{\tau}_k}^{p-1},t_{\boldsymbol
{\tau}_k}^{p}]$ and $t\in(t_{\boldsymbol{\tau}_k}^{n-1},t_{\boldsymbol{\tau}_k}^{n}]$. It follows from \eqref{auxauxdist} and triangular inequality that
\begin{eqnarray*}
d(\overline{U}_{\boldsymbol{\tau}_k}(s),\overline{U}_{\boldsymbol{\tau}_k}(t))&\leq&\int_{t_{\boldsymbol{\tau}_k}^{p-1}}^
{t_{\boldsymbol{\tau}_k}^{n}}{\vert U'_{\boldsymbol{\tau}_k}\vert(r)\ dr}.
\end{eqnarray*}
Letting $k\to+\infty$, and using the weak convergence, we conclude that $u$ is absolutely continuous and $\vert u'\vert\leq m$. Also, after a change of variables, we can employ the identity (\ref{eq:integral-identity}) to obtain
\begin{eqnarray}
\label{eq:teorema-sol-flux-grad}
\frac{1}{2}\int_0^{t_{\boldsymbol{\tau}}^{n}}{\vert U'_{\boldsymbol{\tau}}\vert^2(r)\ dr}+\int_{0}^{t_{\boldsymbol{\tau}}
^{n}}{\frac{d^2(\overline{U}_{\boldsymbol{\tau}}(r),\widetilde{U}_{\boldsymbol{\tau}}(r))}{2r^2}\ dr}
&=&\int_{0}^{t_{\boldsymbol{\tau}}^{n}}{\partial_t\mathcal{E}(r,\widetilde{U}_
{\boldsymbol{\tau}}(r))\ dr}\nonumber\\
&\ &+\mathcal{E}(0,U_{\boldsymbol{\tau}}^0)-\mathcal{E}(t_{\boldsymbol{\tau}}^n,U_{\boldsymbol
{\tau}}^n).
\end{eqnarray}
For the above subsequence, we have $$\mathcal{E}(t,u(t))\leq\displaystyle
\liminf_{k\to\infty}\mathcal{E}(t,\overline{U}_{\boldsymbol{\tau}_k}(t))=\liminf_{k\to\infty}\mathcal{E}
(\mathcal{T}_{\boldsymbol{\tau}_k}(t),\overline{U}_{\boldsymbol{\tau}_k}(t)),$$
and so, using (\ref{eq:est-slope-local}) and (\ref{eq:formula-inclinacao}), we arrive at
\begin{eqnarray}
&\ &\frac{1}{2}\int_0^t\vert u'\vert^2(r)\ dr+\frac{1}{2}\int_0^t\vert \partial \mathcal{E}(r)\vert^2(u(r))\ dr
+\mathcal{E}(t,u(t))\nonumber\\
&\leq&\liminf_{k\to\infty}\Big(\frac{1}{2}\int_0^{\mathcal{T}_{\boldsymbol{\tau}_k}(t)}{\vert U'_{\boldsymbol{\tau}_k}
\vert^2(r)\ dr}+\int_{0}^{\mathcal{T}_{\boldsymbol{\tau}_k}(t)}{\frac{d^2(\overline{U}_{\boldsymbol{\tau}_k}(r),
\widetilde{U}_{\boldsymbol{\tau}_k}(r))}{2r^2}\ dr}\nonumber\\
&\ &\hspace{1cm}+\mathcal{E}(\mathcal{T}_{\boldsymbol{\tau}_k}(t),
\overline{U}_{\boldsymbol{\tau}_k})\Big)\nonumber\\
&\leq&\limsup_{k\to\infty}\int_{0}^{\mathcal{T}_{\boldsymbol{\tau}_k}(t)}{\partial_t\mathcal{E}(r,\widetilde{U}_
{\boldsymbol{\tau}_k}(r))\ dr}+\mathcal{E}(0,u(0))\nonumber\\
&\leq&\int_{0}^{t}{\partial_t\mathcal{E}(r,u(r))\ dr}+\mathcal{E}(0,u_0),\label{eq:curva-inclin-maxim}
\end{eqnarray}
where, by convenience, we have chosen $U_{\boldsymbol{\tau}}^0=u_0$ (recall that $u$ does not depend on $U_{\boldsymbol{\tau}}^0\to u_{0}$). Notice that in particular $\displaystyle\sup_{t\in[0,T]}\mathcal{E}(t,u(t))<\infty$.
On the other hand, in view of \cite[Lemma 1.1.4 a)]{Ambrosio}, there exist an increasing absolutely continuous function $\boldsymbol{s}:[0,T]\to[0,L]$, whose inverse $\boldsymbol{t}$ is Lipschitz, and a curve $\hat{u}:[0,L]\to X$ such that
$\vert\hat{u}'\vert(s)\leq1$ and $u(t)=\hat{u}(\boldsymbol{s}(t))$. Considering the function $\varphi(s)=\mathcal{E}
(\boldsymbol{t}(s),\hat{u}(s))$ and using (\ref{eq:formula-inclinacao}), it follows that
\begin{eqnarray*}
\varphi(s_1)-\varphi(s_2)&\leq&\left(\vert\partial\mathcal{E}(\boldsymbol{t}(s_1))\vert(\hat{u}(s_1))
+\lambda_T^-C\right)\vert s_2-s_1\vert\\
&\ &+(1+C^2)\int_{s_1}^{s_2}{\beta(\boldsymbol{t}(s))\boldsymbol{t}'(s)\ ds},
\end{eqnarray*}
for $s_1<s_2$, where $C=\displaystyle\sup_{s\in[0,L]}d(u^*,\hat{u}(s))$. Replacing the roles of $s_1$ and $s_2$, we obtain
\begin{eqnarray*}
\vert\varphi(s_1)-\varphi(s_2)\vert&\leq& \left(\vert\partial\mathcal{E}(\boldsymbol{t}(s_1))\vert(\hat{u}(s_1))
+\vert\partial\mathcal{E}(\boldsymbol{t}(s_2))\vert(\hat{u}(s_2))+2\lambda_T^-C\right)\vert s_2-s_1\vert\\
&\ &+(1+C^2)\int_{s_1}^{s_2}{\beta(\boldsymbol{t}(s))\boldsymbol{t}'(s)\ ds}.
\end{eqnarray*}
By using Lemma \ref{lem:cond-cont-absolute}, we can conclude that $\varphi$ is absolutely
continuous and then $\mathcal{E}(t,u(t))$ also does so. It follows that $\mathcal{E}(t,u(t))$ is derivable at almost
every point $t\in[0,T]$. Let $t_0\in[0,T]$ be a differentiability point of $\mathcal{E}(t,u(t))$ for which the metric
derivative $\vert u'\vert(t_0)$ exists. Taking $t_n\downarrow t_0$, we get
\begin{eqnarray*}
\frac{d}{dt}\left(\mathcal{E}(t,u(t))\right)\mid_{t=t_0}&\geq&\liminf_{n\to\infty}\frac{\mathcal{E}
(t_n,u(t_n))-\mathcal{E}(t_0,u(t_n))}{t_n-t_0}\\
&\ &+\liminf_{n\to\infty}\frac{\mathcal{E}(t_0,u(t_n))-\mathcal{E}(t_0,u(t_0))}{d(u(t_n),u(t_0))}\frac{d(u(t_n),u(t_0))}{t_n-t_0}\\
&=&\partial_t\mathcal{E}(t_0,u(t_0))-\vert\partial\mathcal{E}(t_0)\vert(u(t_0))\vert u'\vert(t_0).
\end{eqnarray*}
Integrating the above inequality, and using (\ref{eq:curva-inclin-maxim}), we obtain (\ref{eq:energy-identity}).
\end{proof}

\begin{cor}
\label{cor:convergence-qtp} Under the hypotheses of Theorem
\ref{teor:convergence-scheme-discret}. There exists a subsequence of
partitions $\boldsymbol{\tau}_{k}$ such that $\mathcal{E}%
_{\boldsymbol{\tau}_{k}}(t)$ defined in \eqref{eq:functional-interpolation}
converges to $t\rightarrow\mathcal{E}(t,u(t))$ in $L_{loc}^{1}([0,\infty)),$
and therefore a.e. in $\lbrack0,\infty)$ (up to a subsequence), where $u$ is as in
Theorem \ref{teor:convergence-scheme-discret}.
\end{cor}

\begin{proof}
We only need to show that, for $T>0$, the functions $f_{\boldsymbol{\tau}}$ and $g_{\boldsymbol{\tau}}$ defined as
$$f_{\boldsymbol{\tau}}(t):=\mathcal{E}(t_{\boldsymbol{\tau}}^n,U_{\boldsymbol{\tau}}^n), \text{ \ for \ } t\in(t_{\boldsymbol{\tau}}^{n-1},t_{\boldsymbol{\tau}}^{n}],$$
and
$$g_{\boldsymbol{\tau}}(t):=\mathcal{E}(t_{\boldsymbol{\tau}}^{n-1},U_{\boldsymbol{\tau}}^{n-1}), \text{ \ for \ } t\in(t_{\boldsymbol{\tau}}^{n-1},t_{\boldsymbol{\tau}}^{n}],$$
converge to $t\rightarrow\mathcal{E}(t,u(t))$ in $L^1([0,T])$, as $|\boldsymbol{\tau}|\rightarrow 0$. First, note that for each partition $\{0=t_0<t_1<\cdots<t_L=T\}$ of $ [0,T]$, we can bound the variation of $f_{\boldsymbol{\tau}}$ as
\begin{eqnarray}
\sum_{l=1}^{L}{\vert f_{\boldsymbol{\tau}}(t_{l})-f_{\boldsymbol{\tau}}(t_{l-1})\vert}&\leq&\sum_{n=1}^N{\left(
(\mathcal{E}(t_{\boldsymbol{\tau}}^{n},U_{\boldsymbol{\tau}}^{n-1}))-\mathcal{E}
(t_{\boldsymbol{\tau}}^{n},U_{\boldsymbol{\tau}}^{n})\right)}\\
&\ &+C\int_{0}^{T+\tau^*}\beta(s)\ ds,\label{aux-est-4}
\end{eqnarray}
where $C>0$ is a constant independent of $\boldsymbol{\tau}$. By Lemma \ref{lemma:limit-sec-discret}, the summation in the right hand side of (\ref{aux-est-4}) is bounded, and therefore the total variation of $f_{\boldsymbol{\tau}}$ in $[0,T]$ is uniformly bounded. Analogously, the total variation of $g_{\boldsymbol{\tau}}$ in $[0,T]$ is uniformly bounded. It follows from \cite[Chap. 5, Theorem 4]{Evans} that there exist a subsequence $\boldsymbol{\tau}_k$ and functions $A,B\in L^1([0,T])$ such that $f_{\boldsymbol{\tau}_k}\to A$ and $g_{\boldsymbol{\tau}_k}\to B$ in $L^1([0,T])$, as $k\to\infty$. Also, it is not hard to show that $A=B\geq\mathcal{E}(t,u(t))$ a.e in $[0,T]$. Now, the same argument used in the proof of (\ref{eq:energy-identity}) can be used in order to show the equality $A=\mathcal{E}(t,u(t))$ a.e. $t\in[0,T]$. In fact, if $f_{\boldsymbol{\tau}_k}(t)\to A$ in $L^1([0,T])$, then
\begin{eqnarray*}
&\ &\frac{1}{2}\int_0^t\vert u'\vert^2(r)\ dr+\frac{1}{2}\int_0^t\vert \partial \mathcal{E}(r)\vert^2(u(r))\ dr
+A(t)\\
&\leq&\liminf_{k\to\infty}\left(\frac{1}{2}\int_0^{\mathcal{T}_{\boldsymbol{\tau}_k}(t)}{\vert U'_{\boldsymbol{\tau}_k}
\vert^2(r)\ dr}+\int_{0}^{\mathcal{T}_{\boldsymbol{\tau}_k}(t)}{\frac{d^2(\overline{U}_{\boldsymbol{\tau}_k}(r),
\widetilde{U}_{\boldsymbol{\tau}_k}(r))}{2r^2}\ dr}\right.\nonumber\\
&\ &\left.\hspace{1cm}+\mathcal{E}(\mathcal{T}_{\boldsymbol{\tau}_k}(t),
\overline{U}_{\boldsymbol{\tau}_k})\right)\nonumber\\
&\leq&\int_{0}^{t}{\partial_t\mathcal{E}(r,u(r))\ dr}+\mathcal{E}(0,u_0),
\end{eqnarray*}
and, using (\ref{eq:energy-identity}), we are done.
\end{proof}

\begin{rem}
\label{remark:solution-in-localslope} As a consequence, we have that the
solution $u(t)\in\text{Dom}(\vert\partial\mathcal{E}(t)\vert)$ for almost every point $t\in(0,\infty)$.
\end{rem}

\subsection{Contraction property}

\label{subsect:contraction-property}\hspace{0.5cm}Consider the condition

\begin{description}
\item[E6.-] The function $\lambda(t)$ is continuous.
\end{description}

Having at hand the estimates obtained in previous sections, the contraction property holds if we assume \textbf{E6}. Here we only sketch its proof for the reader convenience.

For $\lambda(t)$ continuous, the interpolation $\lambda_{\boldsymbol{\tau}}$
defined in (\ref{eq:lambda-discret}) converges uniformly to $\lambda$, as
$|\boldsymbol{\tau}|\rightarrow0,$ in each bounded interval of $[0,\infty)$.
Recall the following technical lemma \cite[Lemma 23.28]{Villani0}.

\begin{lem}
\label{lem:villani} Let $F=F(t,s)$ be a function $[0,\infty)\times
\lbrack0,\infty)\rightarrow\mathbb{R}$ locally absolutely continuous in the
variable $t$ and uniformly continuous in $s$, and locally absolutely
continuous in $s$ and uniformly in $t$; that is, there exists a nonnegative
$m\in L_{loc}^{1}([0,\infty))$ such that
\[
|F(t,s)-F(t^{\prime},s)|\leq\int_{t^{\prime}}^{t}{m(r)\ dr}\text{ and
}|F(t,s)-F(t,s^{\prime})|\leq\int_{s^{\prime}}^{s}{m(r)\ dr},
\]
where $m$ does not depend on $s$ in the first inequality and on $t$ in the
second one. Then, the function $\delta(t):=F(t,t)$ is locally absolutely
continuous and, for almost every point $t_{0}\in\lbrack0,\infty)$, we have
\begin{equation}
\frac{d}{dt}|_{t=t_{0}}\delta(t)\leq\limsup_{t\uparrow t_{0}}\left(
\frac{F(t_{0},t)-\delta(t_{0})}{t-t_{0}}\right)  +\limsup_{t\downarrow t_{0}%
}\left(  \frac{F(t,t_{0})-\delta(t_{0})}{t-t_{0}}\right)  .
\label{eq:lem-villani}%
\end{equation}

\end{lem}

Integrating (\ref{eq:desig-var-interp}) from $s$ to $t$ with $0\leq s<t\leq T$
and taking the subsequence $\boldsymbol{\tau}_{k}$ given in Corollary
\ref{cor:convergence-qtp}, we can pass the limit and use \textbf{E6} in order
to obtain the inequality
\begin{equation}
\frac{1}{2}d^{2}(u(t),V)-\frac{1}{2}d^{2}(u(s),V)+\int_{s}^{t}{\frac
{\lambda(r)}{2}d^{2}(u(r),V)+\mathcal{E}(r,u(r))\ dr}\leq\int_{s}%
^{t}{\mathcal{E}(r,V)}dr. \label{eq:convex-reform1}%
\end{equation}
Let $u,v$ be two solutions given by Theorem
\ref{teor:convergence-scheme-discret} with initial data $u_{0},v_{0}%
\in\mathbf{D}$, respectively. Recall that, by Lemma
\ref{lemma:limit-sec-discret}, both curves $u$ and $v$ are locally bounded.
Also,
\begin{align*}
|d^{2}(u(t),v(s))-d^{2}(u(t^{\prime}),v(s))|  &  \leq d(u(t),u(t^{\prime
}))(d(u(t),v(s))+d(u(t^{\prime}),v(s)))\\
&  \leq C(T)\int_{t^{\prime}}^{t}{|u^{\prime}|(r)\ dr}.
\end{align*}
where $C(T)=2\displaystyle\sup_{0\leq t\leq T}(d(u(t),u^{\ast})+d(v(t),u^{\ast
}))$. Similarly, one can show the local absolute continuity in the variable
$s$ for the function $F(t,s)=d^{2}(u(t),v(s))$. It follows that $d^{2}%
(u(t),v(s))$ verifies the hypotheses in Lemma \ref{lem:villani}. Next, using
\eqref{eq:convex-reform1}, a direct computation gives
\begin{equation}
\frac{d}{dt}d^{2}(u(t),v(t))+2\lambda(t)d^{2}(u(t),v(t))\leq0,
\label{eq:convex-refom2}%
\end{equation}
for almost every point $t\in\lbrack0,\infty)$, which implies
\begin{equation}
d(u(t),v(t))\leq e^{-\int_{0}^{t}{\lambda(s)\ ds}}d(u_{0},v_{0}).
\label{eq:contraction-property}%
\end{equation}

\begin{rem}
\label{contract-1} The time-dependent functional $\mathcal{E}$ can be
\textquotedblleft weakly\textquotedblright\ convex ($\lambda(t)<0$) at a
certain $t=t_{0}$. In fact, we could have $\int_{0}^{t_{0}}{\lambda(s)\ ds<0}$
and solutions distance themselves. However, according to the behavior of
$\lambda(t)$, the convexity could be improved ($\lambda(t)>0$ and $\int
_{0}^{t}{\lambda(s)\ ds>0}$) as $t$ increases. In this case, we would
recover the time-exponential approximation between the solutions $u$ and $v.$
\end{rem}

\section{Applications for PDEs in the Wasserstein space}

\label{sect:applications-wasserstein}\hspace{0.5cm}In this section we apply
the theory developed in previous ones for time-dependent functionals associated to PDEs in the Wasserstein space. This space has a very nice geometric structure and is suitable to address gradient flow equations.

We start by recalling some definitions and properties of that space. We denote
by $\mathscr{P}_{2}(\mathbb{R}^{d})$ the set of Borel probability measures in
$\mathbb{R}^{d}$ with finite second order moment, i.e. $\mu\in\mathscr{P}_{2}%
(\mathbb{R}^{d})$ if $\mu$ is a positive Borel measure,
\[
\mu(\mathbb{R}^{d})=1\text{ and }M_{2}(\mu):=\int_{\mathbb{R}^{d}}%
{|x|^{2}\ d\mu(x)<\infty}.
\]
We can endow $\mathscr{P}_{2}(\mathbb{R}^{d})$ with the weak-topology or the
so-called narrow topology by considering the following notion of convergence:
\begin{equation}
\mu_{k}\rightharpoonup\mu\text{ as }k\rightarrow\infty\Leftrightarrow
\lim_{k\rightarrow\infty}\int_{\mathbb{R}^{d}}{f(x)}\ d\mu_{k}(x)=\int
_{\mathbb{R}^{d}}{f(x)}\ d\mu(x), \label{eq:narrow-topology}%
\end{equation}
for all $f\in C_{b}^{0}(\mathbb{R}^{d})$, where $C_{b}^{0}(\mathbb{R}^{d})$
stands for the set of bounded continuous functions. On the other hand,
$\mathscr{P}_{2}(\mathbb{R}^{d})$ endowed with the Wasserstein distance is a
complete metric space. This metric is defined by means of the
Monge-Kantorovich problem and reads as%

\begin{equation}
\mathbf{d}_{2}^{2}(\mu,\nu)=\min\left\{  \int_{\mathbb{R}^{d}\times
\mathbb{R}^{d}}{|x-y|^{2}\ d\gamma(x,y)}:\gamma\in\Gamma(\mu,\nu)\right\}  ,
\label{eq:wasserstein-metric}%
\end{equation}
where $\Gamma(\mu,\nu)=\left\{  \gamma\in\mathscr{P}(\mathbb{R}^{d}%
\times\mathbb{R}^{d}):\gamma(A\times\mathbb{R}^{d})=\mu(A),\ \gamma
(\mathbb{R}^{d}\times B)=\nu(B)\right\}  $. In fact, there exists at least one
probability measure in $\mathscr{P}(\mathbb{R}^{d}\times\mathbb{R}^{d})$ that
reaches the minimum in (\ref{eq:wasserstein-metric}). This is called the
optimal transport plane and is supported in the graphic of the subdifferential
of a convex lower semicontinuous function (see \cite{Villani}). We denote by
$\mathscr{P}_{2,ac}(\mathbb{R}^{d})$ the set of probability measures in
$\mathscr{P}_{2}(\mathbb{R}^{d})$ that are absolutely continuous with respect
to the Lebesgue measure. If $\mu$ does not give mass to sets with
Hausdorff-dimension less than $d-1$ (e.g. if $\mu\in\mathscr{P}_{2,ac}%
(\mathbb{R}^{d})$), there exists a map $\mathbf{t}_{\mu}^{\nu}:\mathbb{R}%
^{d}\rightarrow\mathbb{R}^{d}$ that coincides with the gradient of a convex
lower semicontinuous function, such that $\nu=\mathbf{t}_{\mu}^{\nu}\#\mu$ and
the optimal transport plane $\gamma_{0}$ is given by the push-forward of $\mu$
via the map $Id\times\mathbf{t}_{\mu}^{\nu}$, i.e. $\gamma_{0}=(Id\times
\mathbf{t}_{\mu}^{\nu})\#\mu.$ Thus, we have that (see \cite[theorem
2.12]{Villani})
\begin{equation}
\mathbf{d}_{2}^{2}(\mu,\nu)=\int_{\mathbb{R}^{d}}{|x-\mathbf{t}_{\nu}^{\nu
}(x)|^{2}\ d\mu(x)}. \label{eq:wasserstein-metric-monge}%
\end{equation}
We also recall the concept of generalized geodesics \cite{Ambrosio}.

\begin{defn}
\label{def:ganaralized geodesics} Let $\sigma,\mu_{0},\mu_{1}\in
\mathscr{P}_{2}(\mathbb{R}^{d})$ and let $\gamma_{0},\gamma_{1}\in
\mathscr{P}(\mathbb{R}^{d}\times\mathbb{R}^{d})$ be two optimal plans that
reach the minimum in (\ref{eq:wasserstein-metric}) for $d(\sigma,\mu_{0})$ and
$d(\sigma,\mu_{1}),$ respectively. Let $\boldsymbol{\gamma}\in
\mathscr{P}(\mathbb{R}^{d}\times\mathbb{R}^{d}\times\mathbb{R}^{d})$ be a
3-plane such that $P_{1,2}\#\boldsymbol{\gamma}=\gamma_{0}$ and $P_{1,3}%
\#\boldsymbol{\gamma}=\gamma_{1}$ where $P_{i,j}$ denotes the projections on
the coordinates $x_{i}$ and $x_{j}$. A generalized geodesic with base point
$\sigma$ connecting $\mu_{0}$ to $\mu_{1}$ is defined by $\mu_{t}%
=((1-t)P_{2}+tP_{3})\#\boldsymbol{\gamma},$ for $t\in\lbrack0,1]$.
\end{defn}

Although the theory in previous sections can be used to analyze general
functionals in $\mathscr{P}_{2}(\mathbb{R}^{d})$, we shall concentrate our
attention in the following cases:

\textbf{The time-dependent potential energy}
\begin{equation}
\mathcal{V}(t,\mu)=\int_{\mathbb{R}^{d}}{V(t,x)\ d\mu(x)},
\label{eq:potential-energy-td}%
\end{equation}
where $V:[0,\infty)\times\mathbb{R}^{d}\rightarrow\mathbb{R}$ is a
time-dependent potential and \textbf{the time-dependent interaction energy}%

\begin{equation}
\mathcal{W}(t,\mu)=\frac{1}{2}\int_{\mathbb{R}^{d}\times\mathbb{R}^{d}%
}{W(t,x,y)}\ d(\mu\times\mu)(x,y), \label{eq:interaction-energy-td}%
\end{equation}
where $W:[0,+\infty)\times\mathbb{R}^{d}\times\mathbb{R}^{d}\rightarrow
\mathbb{R}$ is an interaction potential. We also are interested in the case of
\textbf{time-dependent diffusion coefficient} in the internal energy functional
\begin{equation}
\mathcal{U}(t,\mu)=\kappa(t)\mathcal{U}(\mu)=\kappa(t)\int_{\mathbb{R}^{d}%
}{\rho\log(\rho)\ dx}, \label{eq:entropy-td}%
\end{equation}
where $\kappa:[0,\infty)\rightarrow(0,\infty)$ and $d\mu=\rho\ dx$ is an
absolutely continuous measure with respect to the Lebesgue one. For singular
measures, we set $\mathcal{U}(t,\mu)=+\infty.$

\begin{rem}
\label{aux-appl-1}The tools developed in previous sections are not directly
applicable for time-dependent $\kappa$. The reason is that the condition
{\bf{E3}} is not satisfied for arbitrary $\kappa$, but only for $\kappa$
constant. So, we postpone the case of $\kappa$ depending on $t$ for later.
\end{rem}

\subsection{The case with constant diffusion}

\label{subsect:constant-diffusion}\hspace{0.5cm}We consider the functionals%

\[
\mathcal{E}_{1}(t,\mu)=\kappa\mathcal{U}(\mu)+\mathcal{V}(t,\mu)
\]
and
\[
\mathcal{E}_{2}(t,\mu)=\mathcal{W}(t,\mu),
\]
where $\mathcal{U}$ is defined in \eqref{eq:entropy-td} and $\kappa\geq0$ is a
constant. In order to apply the theory, we assume some conditions on the potentials.

\begin{description}
\item[V1.-] For each fixed $t\geq0$, $V(t,\cdot)$ is $\lambda(t)-$convex, for
some function $\lambda:[0,\infty)\rightarrow\mathbb{R}$ in $L_{loc}^{\infty
}([0,\infty))$, that is, $V(t,x)-\frac{\lambda(t)}{2}|x|^{2}$ is convex.

\item[V2.-] Let $\partial^{\circ}V(t,x)$ denote the element of minimal norm in
the subdifferential of $V(t,\cdot)$ at the point $x\in\mathbb{R}^{d}$. We
assume that $|\partial^{\circ}V(t,0)|$ is locally bounded and $t\rightarrow
V(t,0)$ is locally bounded from below.

\item[V3.-] There exists a function $\beta\in L_{loc}^{1}([0,+\infty))$ such
that
\begin{equation}
|V(s,x)-V(t,x)|\leq\int_{s}^{t}{\beta(r)\ dr}(1+|x|^{2}),\text{ for }0\leq
s<t. \label{eq:V3}%
\end{equation}

\end{description}

We consider \textbf{V2} for $x=0$ only for simplicity. Indeed, this condition
can be assumed for any (fixed) $x_{0}\in\mathbb{R}^{d}$. Moreover, it is not
necessary to choose the element of minimal norm in the subdifferential. In fact, it would be enough to make a measurable choice (in $t$) in the subdifferential.

We start with the following result.

\begin{prop}
\label{prop:Hip-V} Assume the hypotheses \textbf{V1} to \textbf{V3}. If there
exists $\mu\in\mathscr{P}_{2}(\mathbb{R}^{d})$ such that $d\mu=\rho dx$,
\begin{equation}
\int_{\mathbb{R}^{d}}{\rho\log(\rho)\ dx}<\infty,\text{ and }\int
_{\mathbb{R}^{d}}{V(0,x)\ d\mu(x)}<\infty, \label{eq:prop-hip-V-aux}%
\end{equation}
then the functional $\mathcal{E}_{1}$ satisfies \textbf{E1} to \textbf{E5}.
\end{prop}

\begin{proof}
Taking $s=0$ in {\bf V3}, it follows that
\begin{eqnarray}
\label{eq:prop-hip-V-0}
\vert V(t,x)-V(0,x)\vert&\leq&\int_0^{t}{\beta(s)\ ds}(1+\vert x\vert^2).
\end{eqnarray}
Next, we show the estimate
\begin{equation}
\label{eq:hip-V-1}
V(t,x)\geq -A(t)-B_T\vert x\vert^2
\end{equation}
where $A(t)=-V(t,0)+\frac{1}{2}\vert\partial^{\circ}V(t,0)\vert^2$ and  $B_T=\frac{1}{2}(1+\lambda_T^-)$ for all $t
\in[0,T]$.
In fact, by the definition of subdifferential, we have
\begin{eqnarray*}
V(t,x)&\geq& V(t,0)+\langle \partial^{\circ}V(t,0),x\rangle -\frac{\lambda_T^-}{2}\vert x\vert^2\\
&\geq& V(t,0)-\frac{1}{2}\vert \partial^{\circ}V(t,0)\vert^2-\frac{1}{2}(1+\lambda_T^-)\vert x\vert^2,
\end{eqnarray*}
and so \eqref{eq:hip-V-1} follows. This estimate implies that the functional $\mathcal{V}$ is
lower semicontinuous with respect to the Wasserstein metric, for each fixed $t\geq0$. Since the internal energy functional
$\mathcal{U}$ is also lower semicontinuous (see \cite{Villani}), we obtain {\bf E1}. Using \eqref{eq:prop-hip-V-0}, the second condition in (\ref{eq:prop-hip-V-aux}), and \eqref{eq:hip-V-1}, it follows that $\text{Dom}(\mathcal{V}(t,\cdot))$ is nonempty and time-independent, which gives {\bf E2}. The property {\bf E3} is a direct consequence of {\bf V3} by taking $u^*=\delta_0\in\mathscr{P}_2(\mathbb{R}^d)$. {\bf E5} follows by using the convexity of the function $\mu\to\mathbf{d}_2^2(\sigma,\mu)$ along generalized geodesics with base point $\sigma$ and the convexity of
the potential $V(t,\cdot)$.
Next, we turn to {\bf E4}. Recall the estimate \cite{Jordan-Otto}
\begin{equation}
\label{eq:otto-estimate}
\int_{\mathbb{R}^d}{\rho\log(\rho)\ dx}\geq -C(1+M_2(\mu))^{\alpha},
\end{equation}
where $\alpha\in(0,1)$ and $C>0$ are constants depending only on the dimension $d$, and $d\mu=\rho dx\in\mathscr{P}_{2,ac}(\mathbb{R}^d)$. Thus, we obtain from (\ref{eq:hip-V-1}) that
\begin{eqnarray*}
\frac{\mathbf{d}^2_2(\delta_0,\mu)}{2\tau^*}+\mathcal{E}_1(t,\mu)&\geq& \frac{1}{2\tau^*}M_2(\mu)-\kappa C(1+M_2(\mu)
)^{\alpha}-A(t)-BM_2(\mu)\\
&=&\left(\frac{1}{2\tau^*}-B\right)M_{2}(\mu)-\kappa C(1+M_2(\mu))^{\alpha}-A(t).
\end{eqnarray*}
Choosing $\tau^*(T)>0$ such that $\frac{1}{\tau^*(T)}>1+\lambda_T^-$, and using {\bf V2}, the last expression is bounded from below
by a constant depending on $\alpha,\kappa,d,\lambda_T^-,\tau^*, T$, and so {\bf E4} follows.
\end{proof}

In view of the hypotheses in Theorem \ref{theor:sol-flux-grad}, we need to
impose one more condition on $V$ in order to obtain the needed regularity for
the functional $\mathcal{V}$, as expected.

\begin{lem}
\label{lem:t-diff-V} Let $\kappa\geq0$ and $\mathbf{D_{1}}=\mbox{Dom}(\mathcal{E}_{1})$. If, in addition to \textbf{V1, V2 }and\textbf{ V3}, we
assume that $t\rightarrow V(t,x)$ is differentiable for each $x\in
\mathbb{R}^{d}$, then the function $t\rightarrow\mathcal{V}(t,\mu)$ is
differentiable for each $\mu\in\mathbf{D_{1}.}$ Moreover, for each sequence
$t_{n}\rightarrow t$ and $\mathbf{d}_{2}(\mu_{n},\mu)\rightarrow0,$ we have
that
\begin{equation}
\lim_{n\rightarrow\infty}\frac{\mathcal{V}(t_{n},\mu_{n})-\mathcal{V}%
(t,\mu_{n})}{t_{n}-t}=\int_{\mathbb{R}^{d}}{\frac{\partial}{\partial
t}V(t,x)\ d\mu(x).} \label{eq:deriv-V}%
\end{equation}

\end{lem}

\begin{proof}
We take $\sigma\in\mathscr{P}_{2,ac}(\mathbb{R}^d)$ and the maps $\mathbf{t}_{\sigma}^{\mu_n}$ and $\mathbf{t}_{\sigma}^{\mu}$ that realize the optimal transports from $\sigma$ to $\mu_n$ and from $\sigma$ to $\mu$, respectively. Then,
\begin{eqnarray}
\frac{\mathcal{V}(t_n,\mu_n)-\mathcal{V}(t,\mu_n)}{t_n-t}=\int_{\mathbb{R}^d}{\frac{V(t_n,\mathbf{t}_{\sigma}^{\mu_n})
-V(t,\mathbf{t}_{\sigma}^{\mu_n})}{t_n-t}\ d\sigma(x)}. \label{aux-est-5}
\end{eqnarray}
By \cite[pag. 71]{Villani}, we have that $\mathbf{t}_{\sigma}^{\mu_n}(x)\to\mathbf{t}_{\sigma}^{\mu}(x)$ a.e. in $\mathbb{R}^d$ with respect to $\sigma$. Using a version of the dominated convergence theorem, we can take
the limit in (\ref{aux-est-5}), as $n\to\infty$, and obtain (\ref{eq:deriv-V}).
\end{proof}

The metric space $\mathscr{P}_{2}(\mathbb{R}^{d})$ and functionals addressed
here present more structure than those in previous sections, where an abstract
theory has been developed. So, it is natural to wonder if gradient flow
solutions as in Definition \ref{defn:def-solution} is related to other senses
of solutions in$\mathscr{P}_{2}(\mathbb{R}^{d}).$ In this direction, we show
that the solution $u$ associated to the functional $\mathcal{E}_{1}$ is in fact
a distributional solution for the Fokker-Planck equation. In the next result, we state precisely
this fact and give some properties for $u$.

\begin{teor}
\label{teor:Fokk-Planck-eq} Consider the functional $\mathcal{E}_{1}$ with
$\kappa\geq0$ and potential $V$ satisfying the assumptions \textbf{V1} to
\textbf{V3} and the differentiability condition in Lemma \ref{lem:t-diff-V}.
Then, given $\mu_{0}\in\mathscr{P}_{2}(\mathbb{R}^{d})$, the curve
$\mu:[0,\infty)\rightarrow\mathscr{P}_{2}(\mathbb{R}^{d})$ given in Theorem
\ref{teor:convergence-scheme-discret} is a distributional solution for the
Fokker-Planck equation
\begin{equation}
\partial_{t}\rho=\kappa\Delta\rho+\nabla\cdot(\nabla V(t,x)\rho),
\label{eq:fokker-planck-1}%
\end{equation}
with $\lim_{t\rightarrow0^{+}}\mu(t)=\mu_{0}$ weakly as measure. If $\kappa
>0$, such curve is absolutely continuous with respect to the Lebesgue measure,
i.e. $d\mu_{t}(x)=\rho(t,x)dx$, and $\rho(t,\cdot)\in W_{loc}^{1,1}(\mathbb{R}^{d})$. Also, $\mu$ satisfies the energy identity
\begin{equation}
\mathcal{E}_{1}(s,\mu(s))=\mathcal{E}_{1}(t,\mu(t))+\int_{s}^{t}%
{\int_{\mathbb{R}^{d}}{(|\Psi_{1}(r,t)|^{2}-\partial_{t}V(r,x))\ d\mu_{r}%
(x)}\ dr} \label{eq:energy-fokker-planc}%
\end{equation}
for $s<t$, where $\Psi_{1}:[0,\infty)\times\mathbb{R}^{d}\rightarrow
\mathbb{R}^{d}$ is a vector field satisfying the identity
\begin{equation}
\rho(t,x)\Psi_{1}(t,x)=\kappa\nabla\rho(t,x)+\rho(t,x)\nabla_{x}V(t,x),\; for\;\kappa>0,
\label{eq:fokker-planck-vector-field}%
\end{equation}
and $\Psi_{1}=\partial^{\circ}V(t,x)$ for $\kappa=0$. Moreover, if the function $\lambda$ satisfies {\bf{E6}}, and $\mu_{1}%
,\mu_{2}$ are two solutions, we have the contraction property
\begin{equation}
\mathbf{d}_{2}(\mu_{1}(t),\mu_{2}(t))\leq e^{-\int_{0}^{t}\lambda
(s)\ ds}\mathbf{d}_{2}(\mu_{0},\mu_{1}). \label{eq:fokker-planck-contraction}%
\end{equation}

\end{teor}

\begin{proof}
First we calculate the variation of $\mathcal{E}_1(t,\mu(t))$. We have that
\begin{eqnarray}
\mathcal{E}_1(s,\mu(s))-\mathcal{E}_1(t,\mu(t))&=&\kappa\left(\mathcal{U}(\mu(s))-\mathcal{U}(\mu(t))\right)
+\int_{\mathbb{R}^d}{\left(V(s,x)-V(t,x)\right)\ d\mu_s}\\
&\ &+\int_{\mathbb{R}^d}{V(t,x)\ d\mu_s}-\int_{\mathbb{R}^d}{V(t,x)\ d\mu_t}.\label{aux-est-6}
\end{eqnarray}
Dividing (\ref{aux-est-6}) by $s-t$, using Lemma \ref{lem:t-diff-V}, and recalling that the function $\mathcal{E}_1(t,\mu(t))$ is absolutely
continuous, we get
\begin{equation}
\label{eq:formula-derivative-E1}
\frac{d}{dt}\mathcal{E}_1(t,\mu(t))=-\int_{\mathbb{R}^d}{\langle\Psi_1(t,x),v(t,x)\rangle\ d\mu_t(x)}
+\int_{\mathbb{R}^d}{\partial_tV(t,x)\ d\mu_t(x)},
\end{equation}
where $v:[0,\infty)\times\mathbb{R}^d\to\mathbb{R}^d$ is the vector field associated to the absolutely continuous curve $\mu_t$, $\lVert v(t,\cdot)\rVert_{L^2(\mu_t;\mathbb{R}^d)}=\vert\mu'\vert(t)$, and $\Psi_1$ is the vector field satisfying $\lVert \Psi_1(t)\rVert_{L^2(\mu_t;\mathbb{R}^d)}=\vert\partial\mathcal{E}_1(t)\vert(\mu_t)$. Moreover, $v$ verifies the continuity equation
\begin{equation}
\partial_t\mu_t+\nabla\cdot(v_t\mu_t)=0,
\end{equation}
in the distributional sense, with $v_t(x)=v(t,x)$. Using (\ref{eq:energy-identity}) together with (\ref{eq:formula-derivative-E1}), we obtain $-\Psi_1(t,x)=v(t,x)$ for $\mu_t$-a.e. $x\in\mathbb{R}^d$ and the identity \eqref{eq:energy-fokker-planc}.
\end{proof}

Similar results hold true for the functional $\mathcal{E}_{2}$. In what
follows, we state the hypotheses for $W$ and $\mathcal{W}$ in order to treat
$\mathcal{E}_{2}$ in light of the abstract Theorem \ref{theor:sol-flux-grad}
in metric spaces.

\begin{description}
\item[W1.-] For each fixed $t\geq0$, the interaction potential $W(t,x,y)$ is
symmetric and, for $t=0$, it satisfies a quadratic growth condition, namely
$W(t,x,y)=W(t,y,x)$ and $W(0,x,y)\leq C(1+|x|^{2}+|y|^{2})$.

\item[W2.-] For each fixed $t\geq0$, $W(t,\cdot)$ is $\lambda(t)-$convex, for
some function $\lambda:[0,\infty)\rightarrow\mathbb{R}$ as in \textbf{E5}. Let
$\partial^{\circ}W(t,x,y)$ denote the element of minimal norm in the
subdifferential of $W(t,\cdot)$ at the point $(x,y)\in\mathbb{R}^{d}%
\times\mathbb{R}^{d}$. We assume that $|\partial^{\circ}W(t,0,0)|$ is locally
bounded and $t\rightarrow W(t,0,0)$ is locally bounded from below.

\item[W3.-] There exists a function $\beta\in L_{loc}^{1}([0,+\infty))$ such
that
\begin{equation}
|W(s,x,y)-W(t,x,y)|\leq\int_{s}^{t}{\beta(r)\ dr}(1+|x|^{2}+|y|^{2}),\text{
for }0\leq s<t. \label{eq:W3}%
\end{equation}

\end{description}

The reason for assuming \textbf{W1} is to obtain a quadratic growth for
$W(t,x,y),$ for each $t>0,$ and then one can use the results in
\cite{Carrillo-lisini-mainini}. In fact, using \textbf{W1}, this growth
follows directly from \textbf{W3}. Proceeding as in Proposition
\ref{prop:Hip-V}, again we get that the functional $\mathcal{E}_{2}$ satisfies
\textbf{E1} to \textbf{E5}. Assuming a differentiability property in the
$t$-variable, we obtain the analogous of Lemma \ref{lem:t-diff-V}. Here we
only state the results for the functional $\mathcal{E}_{2}$. The proof is
similar to that of Theorem \ref{teor:Fokk-Planck-eq} and is left to the reader.

\begin{teor}
\label{teor:Mckean-Vlasov-equation} Consider the functional $\mathcal{E}_{2}$
with the interaction potential $W$ satisfying \textbf{W1} to \textbf{W3}.
Suppose also that for $(x,y)\in\mathbb{R}^{d}\times\mathbb{R}^{d}$ the
function $t\rightarrow W(t,x,y)$ is differentiable. Then, given $\mu_{0}%
\in\mathscr{P}_{2}(\mathbb{R}^{d})$, the curve $\mu:[0,\infty)\rightarrow
\mathscr{P}_{2}(\mathbb{R}^{d})$ given in Theorem
\ref{teor:convergence-scheme-discret} is a distributional solution for the
continuity equation
\begin{equation}
\partial_{t}\rho=\nabla\cdot(\mathbf{v}(t,x)\rho), \label{eq:mckean-Vlasov-1}%
\end{equation}
with $\lim_{t\rightarrow0^{+}}\mu(t)=\mu_{0}$ weakly as measure, where
\begin{equation}
\mathbf{v}(t,x)=\int_{\mathbb{R}^{d}}{\eta(t,x,y)\rho(y)\ dy},\text{ \ }%
\mu\text{-a.e. in \ }\mathbb{R}^{d},
\label{eq:vector-field-mckean-vlasov}%
\end{equation}
and $\eta(t,x,y)=\frac{1}{2}(\eta_{1}(t,x,y)+\eta_{2}(t,y,x))$ for some Borel
measurable selection $(\eta_{1},\eta_{2})\in\partial W(t,\cdot,\cdot)$.
Moreover, $\mu$ satisfies the energy identity
\begin{equation}
\mathcal{E}_{2}(s,\mu(s))=\mathcal{E}_{2}(t,\mu(t))+\int_{s}^{t}%
{\int_{\mathbb{R}^{d}}{(|\mathbf{v}(r,x)|^{2}-\partial_{t}W(r,x))\ d\mu
_{r}(x)}\ dr} \label{eq:energy-mckean-vlasov}%
\end{equation}
for $s<t$. Furthermore, if the function $\lambda$ satisfies {\bf{E6}} and
$\mu_{1},\mu_{2}$ are two solutions, we have the contraction property
\begin{equation}
\mathbf{d}_{2}(\mu_{1}(t),\mu_{2}(t))\leq e^{-\int_{0}^{t}\lambda
(s)\ ds}\mathbf{d}_{2}(\mu_{0},\mu_{1}). \label{eq:mckean-vlasov-contraction}%
\end{equation}

\end{teor}

\begin{rem}
\label{rem-aux-3} Let us observe that the vector field $\mathbf{v}$ in
\eqref{eq:mckean-Vlasov-1} is characterized by the form
\eqref{eq:vector-field-mckean-vlasov} thanks to the results of
Carrillo-Lisini-Mainini \cite{Carrillo-lisini-mainini}. They showed that, in
general, the Borel measurable selection of $\eta$ depends on the probability
$\mu\in\mathscr{P}_{2}(\mathbb{R}^{d})$ and is not necessarily given by the
minimal selection in the subdifferential of $W$. In the particular case when
$W(t,x,y)=w(t,y-x)$ is given by a symmetric function $w:[0,\infty
)\times\mathbb{R}^{d}\rightarrow\mathbb{R}$, the $\lambda(t)$-convexity of $W(t,\cdot
,\cdot)$ follows from the one of $w$ only if $\lambda(t)\leq0$ and therefore we can not use the
results for any $\lambda(t)$-convexity of $w$.

Of course, we can consider a functional of the type (see subsection below to the
time-dependent viscosity)
\[
\mathcal{E}(t,\mu)=\mathcal{U}(\mu)+\mathcal{W}(t,\mu)
\]
and apply the metric theory in order to obtain existence of curves satisfying
the conclusions in Theorem \ref{theor:sol-flux-grad}, the contraction property
and a continuity equation. On the other hand, we do not know how to describe
the velocity field $\mathbf{v}$ in this general case, however it is expect
that $\mu$ satisfies a Mackean-Vlasov equation of the type $\partial_{t}\mu
=\Delta\mu+\nabla\cdot(\mathbf{v}\mu)$ for $\mathbf{v}$ as in \eqref{eq:vector-field-mckean-vlasov}.

Finally, if we assume that $W(t,x,y)=w(t,y-x)$ is $\lambda(t)$-convex with
$\lambda(t)\leq0$ and satisfies a doubling condition property $w(t,x+y)\leq
C_{t}(1+w(x)+w(y))$, then one can show that the curve $\mu_{t}$ given in
Theorem \ref{teor:convergence-scheme-discret} is a distributional solution of
the Mackean-Vlasov equation
\[
\partial_{t}\mu_{t}=\Delta\mu_{t}+\nabla\cdot((\nabla w(t)\ast\mu_{t})\mu
_{t}).
\]

\end{rem}

\subsection{The case with time-dependent diffusion}

\label{subsect:variable-diffusion}\hspace{0.5cm}Now we consider the case when
$\kappa:[0,\infty)\rightarrow(0,\infty).$ For the sake of simplicity, we
consider the functional%

\begin{equation}
\mathcal{E}(t,\mu)=\kappa(t)\mathcal{U}(\mu)+\mathcal{V}(t,\mu),
\label{eq:kappa-time-dependent}%
\end{equation}
where $\mathcal{U}$ and $\mathcal{V}$ are defined in \eqref{eq:entropy-td} and
\eqref{eq:potential-energy-td}, respectively, and $\kappa$ is a positive
function locally absolutely continuous. Also, we assume that \textbf{V1} to
\textbf{V3} hold true. Thus, by assuming that $V(0,\cdot)$ satisfies
\eqref{eq:prop-hip-V-aux}, we have that the domain of $\mathcal{E}$ is
time-independent. Notice that the functional $\mathcal{E}$ satisfies
\textbf{E1}, \textbf{E2}, \textbf{E4} and \textbf{E5}, but not \textbf{E3}. In
fact, as observed in Remark \ref{aux-appl-1}, \textbf{E3 }holds true, if and
only if, $\kappa(t)$ is constant. Here we need to assume that $\kappa$ is
non-increasing. An important fact is the following:

\begin{rem}
\label{rem:entropy-boundedness} Let $\mu_{n}\in\mathscr{P}_{2}(\mathbb{R}%
^{d})$ and $t_{n}\in\lbrack0,\infty)$ be two bounded sequences, where $\mu
_{n}$ is bounded with respect to the Wasserstein metric $\mathbf{d}_{2}$, such
that the numeric sequence $\mathcal{E}(t_{n},\mu_{n})$ is bounded from above.
Then, the numeric sequences $\mathcal{U}(\mu_{n})$ and $\mathcal{V}(t_{n}%
,\mu_{n})$ are bounded. In fact, it follows from \eqref{eq:otto-estimate} that
the sequence $\kappa(t_{n})\mathcal{U}(\mu_{n})$ is bounded from below and
thus $\mathcal{V}(t_{n},\mu_{n})$ is bounded from above. Similarly, from
\eqref{eq:hip-V-1} we have that $\mathcal{V}(t_{n},\mu_{n})$ is bounded from
below, and then $\kappa(t_{n})\mathcal{U}(\mu_{n})$ is bounded from above.
\end{rem}

Using Remark \ref{rem:entropy-boundedness}, we obtain easily the same
conclusions of Lemmas \ref{lemma:existence-minimizer} and
\ref{lem:continuity-minimizer}. Let the potential $V$ be differentiable in the $t$-variable. Using the minimality of $\mu_{\tau}^{t+\tau}$, we have that for
$\tau_{0}<\tau_{1}$
\begin{align*}
\mathscr{E}_{t+\tau_{1},\tau_{1}}(u)-\mathscr{E}_{t+\tau_{0},\tau_{0}}(u)  &
\leq(\kappa(t+\tau_{1})-\kappa(t+\tau_{0}))\mathcal{U}(\mu_{\tau_{0}}%
^{t+\tau_{0}})+\mathcal{V}(t+\tau_{1},u_{\tau_{0}}^{t+\tau_{0}})\\
&  \ -\mathcal{V}(t+\tau_{0},u_{\tau_{0}}^{t+\tau_{0}})+\frac{\tau_{0}%
-\tau_{1}}{2\tau_{1}\tau_{0}}\mathbf{d}_{2}^{2}(u,u_{\tau_{0}}^{t+\tau_{0}})\\
&  \leq(\kappa(t+\tau_{1})-\kappa(t+\tau_{0}))\mathcal{U}(\mu_{\tau_{0}%
}^{t+\tau_{0}})+\frac{\tau_{0}-\tau_{1}}{2\tau_{1}\tau_{0}}\mathbf{d}_{2}%
^{2}(u,u_{\tau_{0}}^{t+\tau_{0}})\\
&  \ +\int_{t+\tau_{0}}^{t+\tau_{1}}{\beta(r)\ dr}(1+M(u_{\tau_{0}}%
^{t+\tau_{0}})),
\end{align*}
where above we used the estimate \eqref{eq:otto-estimate}. Analogously, the
reverse inequality follows. In view of Remark \ref{rem:entropy-boundedness},
we can argue as in Proposition \ref{prop:differentiability-property1} and
Corollary \ref{cor:integral-identity} in order to obtain the identity
\eqref{eq:integral-identity} for the functional (\ref{eq:kappa-time-dependent}).

Up until this point, notice that we have not needed the monotonicity
hypothesis for $\kappa$. In what follows, we comment on an essential step in
order to recover Lemma \ref{lemma:limit-sec-discret}. In fact, recalling the
notation for discrete solution $(U_{\boldsymbol{\tau}}^{j})$ of the
variational scheme, and using that $\kappa(t)$ is non-increasing and
\eqref{eq:otto-estimate}, we estimate
\begin{align*}
\sum_{j=1}^{n}{\kappa(t_{\boldsymbol{\tau}}^{j})(\mathcal{U}%
(U_{\boldsymbol{\tau}}^{j-1})-\mathcal{U}(U_{\boldsymbol{\tau}}^{j}))}  &
\leq\kappa(0)\mathcal{U}(U_{\boldsymbol{\tau}}^{0})-\kappa
(t_{\boldsymbol{\tau}}^{n})\mathcal{U}(U_{\boldsymbol {\tau}}^{n})\\
&  \ -C\sum_{j=1}^{n}{(\kappa(t_{\boldsymbol{\tau}}^{j})-\kappa
(t_{\boldsymbol{\tau}}^{j-1}))}(1+M_{2}(U_{\boldsymbol {\tau}}^{j-1})).
\end{align*}
From here we can repeat the arguments in order to obtain the same conclusion
of Lemma \ref{lemma:limit-sec-discret}. In the case when $\mathcal{V}\equiv0$,
it is not necessary to suppose the monotonicity of $\kappa$ because the
difference $\mathcal{U}(U_{\boldsymbol{\tau}}^{j-1})-\mathcal{U}%
(U_{\boldsymbol{\tau}}^{j})$ is positive and a more direct estimate can be performed.

Going back to Section \ref{sect:apriori-estimates}, it is easy to see that it
remains only to estimate
\begin{equation}
\int_{0}^{t}{[(1-l_{\boldsymbol{\tau}}(s))(\kappa(\mathcal{T}%
_{\boldsymbol{\eta}}(s))-\kappa(\mathcal{T}_{\boldsymbol{\tau}}%
(s)))\mathcal{U}(\underline{U}_{\boldsymbol{\tau}}(s))]^{+}\ ds},\text{ for
}0\leq t\leq T, \label{aux-est-7}%
\end{equation}
where $T>0$ is fixed and $\boldsymbol{\tau}$, $\boldsymbol{\eta}$ are two
partitions with small sizes. Indeed, since $\mathcal{E}(\mathcal{T}%
_{\boldsymbol{\tau}}(t),\overline{U}_{\boldsymbol{\tau}}(t))$ is bounded from
above by a constant independent of $\boldsymbol{\tau}$, it follows from Remark
\ref{rem:entropy-boundedness} that $\mathcal{U}(\underline{U}%
_{\boldsymbol{\tau}}(t))$ is bounded by a constant independent of
$\boldsymbol{\tau}$. Thus, the integral (\ref{aux-est-7}) can be estimated by
proceeding similarly to Proposition \ref{prop:residuo-G}, and then we obtain
the convergence of the approximate solutions \eqref{eq:sol-interpolante1}. In
this way, the functional $\mathcal{E}$ defined in
\eqref{eq:kappa-time-dependent} presents properties and results contained in
Sections \ref{sect:apriori-estimates} and \ref{sect:regularity}. So, we have
the following:

\begin{teor}
\label{teor:kappa-time-depend} Let $\mathcal{E}$ be the functional defined in
\eqref{eq:kappa-time-dependent} with $\kappa:[0,\infty)\rightarrow(0,\infty)$
an non-increasing absolutely continuous function and let the potential $V$
satisfy \textbf{V1} to \textbf{V3} and the differentiability condition in
Lemma \ref{lem:t-diff-V}. Then, given $\mu_{0}\in\mathscr{P}_{2}%
(\mathbb{R}^{d})$, the curve $\mu:[0,\infty)\rightarrow\mathscr{P}_{2}%
(\mathbb{R}^{d})$ obtained in Theorem \ref{teor:convergence-scheme-discret} is
absolutely continuous with respect to the Lebesgue measure, i.e. $d\mu
_{t}(x)=\rho(t,x)dx$, $\rho(t,\cdot)\in W_{loc}^{1,1}(\mathbb{R}^{d})$ for
each $t\in\lbrack0,\infty)$, and $\rho$ is a distributional solution for the
Fokker-Planck equation
\begin{equation}
\partial_{t}\rho=\kappa(t)\Delta\rho+\nabla\cdot(\nabla V(t,x)\rho),
\label{eq:kappa-time-dependent-fokker-planck-1}%
\end{equation}
with $\lim_{t\rightarrow0^{+}}\mu(t)=\mu_{0}$ weakly as measure. Also,
$\mu(t)$ satisfies the energy identity
\begin{align}
\mathcal{E}_{1}(s,\mu(s))  &  =\mathcal{E}_{1}(t,\mu(t))+\int_{s}^{t}%
{\int_{\mathbb{R}^{d}}{(|\Psi_{1}(r,t)|^{2}-\partial_{t}V(r,x))\rho
(r,x)\ dx}\ dr}\nonumber\\
&  \ -\int_{s}^{t}{\int_{\mathbb{R}^{d}}{\kappa^{\prime}(r)\rho(r,x)\log
(\rho(r,x))\ dx}\ dr,}\text{ for }s<t,
\label{eq:kappa-time-dependent-energy-fokker-planc}%
\end{align}
where $\Psi_{1}:[0,\infty)\times\mathbb{R}^{d}\rightarrow\mathbb{R}^{d}$ is a
vector field satisfying
\begin{equation}
\rho(t,x)\Psi_{1}(t,x)=\kappa(t)\nabla\rho(t,x)+\rho(t,x)\nabla V(t,x)\,\text{
for }\mu_{t}\text{-a.e. \ }x\in\mathbb{R}^{d}. \label{eq:kappa-time-dependent-fokker-planck-vector-field}%
\end{equation}
Moreover, if the function $\lambda$ satisfies {\bf{E6}} and $\mu_{1},\mu
_{2}$ are two solutions, we have the contraction property
\begin{equation}
\mathbf{d}_{2}(\mu_{1}(t),\mu_{2}(t))\leq e^{-\int_{0}^{t}\lambda
(s)\ ds}\mathbf{d}_{2}(\mu_{0},\mu_{1}).
\label{eq:kappa-time-dependent-fokker-planck-contraction}%
\end{equation}

\end{teor}

\subsection{More general internal energy}

\label{sub-general}\hspace{0.5cm}In this subsection we give the outline to construct the
time-dependent gradient flow for more general internal energy functionals. Let
$U:[0,\infty)\times\lbrack0,\infty)\rightarrow\mathbb{R}$ be a continuous
function such that $C^{1}((0,\infty)\times(0,\infty))$. Consider the internal
energy functional
\begin{equation}
\mathcal{U}(t,\mu)=\left\{
\begin{array}
[c]{ccc}%
\int_{\mathbb{R}^{d}}{U(t,\rho(x))\ dx,} & \text{ if } & d\mu=\rho\ dx\\
+\infty, & \text{ } & \text{otherwise.}%
\end{array}
\right.  \label{eq:general-internal-energy}%
\end{equation}

We assume the following condition on $U$.

\begin{description}
\item[U1.-] There exist functions $a,A:[0,\infty)\rightarrow\lbrack0,\infty)$
with $a\in L_{loc}^{1}([0,\infty))$ and $A\in L^{1}([0,\infty))$ such that
\begin{equation}
-A(t)U^{+}(0,z)\leq\frac{\partial U}{\partial t}(t,z)\leq a(t)U^{-}(0,z),
\label{eq:U1-cond}%
\end{equation}
for all $t,z\in\lbrack0,+\infty)$, and $U(0,z)$ has superlinear growth at
infinite, i.e. $\displaystyle\lim_{z\rightarrow+\infty}\frac{U(0,z)}%
{z}=+\infty$.

\item[U2.-] There exist $\alpha\in(0,1)$ with $\alpha>\frac{d}{d+2}$ and
positive constants $c_{1},c_{2}\geq0$ such that
\[
U(0,z)\geq-c_{1}z-c_{2}z^{\alpha}.
\]

\item[U3.-] $U(0,0)=0$, $z\rightarrow U(t,z)$ is convex$,$ and $z\rightarrow
z^{d}U(t,z^{-d})$ is convex and non-increasing on $(0,+\infty),$ for each
$t>0$.
\end{description}

Without loss of generality, we can assume that $\Vert A\Vert_{1}=\int
_{0}^{\infty}{A(t)\ dt<1}$; otherwise, we can replace $U$ by $\frac{U}{\Vert
A\Vert_{1}+1}$. Firstly, let us note that \textbf{U1} and \textbf{U2} imply
\begin{align}
U(t,z)  &  =\int_{0}^{t}{\frac{\partial U}{\partial t}(r,z)\ dr}%
+U(0,z)\nonumber\\
&  \geq-\left(  \int_{0}^{t}{A(r)\ dr}\right)  U^{+}(0,z)+U(0,z)\nonumber\\
&  =\left(  1-\int_{0}^{t}{A(r)\ dr}\right)  U^{+}(0,z)-U^{-}%
(0,z)\label{eq:super-lineality}\\
&  \geq-U^{-}(0,z)\geq-c_{1}z-c_{2}z^{\alpha}. \label{eq:under-bound}%
\end{align}
Then, recalling that $\alpha>\frac{d}{d+2}$, it follows from
\eqref{eq:under-bound} that
\begin{align}
\mathcal{U}(t,\mu)  &  \geq-\left(  c_{1}+c_{2}\int_{\mathbb{R}^{d}}%
{\rho(x)^{\alpha}\ dx}\right) \nonumber\\
&  \geq-\left(  c_{1}+c_{2}\left(  \int_{\mathbb{R}^{d}}{(1+|x|^{2}%
)\rho(x)\ dx}\right)  ^{\alpha}\left(  \int_{\mathbb{R}^{d}}{\frac
{1}{(1+|x|^{2})^{\frac{\alpha}{1-\alpha}}}\ dx}\right)  ^{1-\alpha}\right)
\nonumber\\
&  =-\left(  c_{1}+c_{2}C_{\alpha}(1+M_{2}(\mu))^{\alpha}\right)  .
\label{eq:under-bound1}%
\end{align}
Therefore, the functional in \eqref{eq:general-internal-energy} is
well-defined from $[0,+\infty)\times\mathscr{P}_{2}(\mathbb{R}^{d})$ to
$(-\infty,+\infty]$. It follows from \eqref{eq:super-lineality} that
$U(t,\cdot)$ has a superlinear growth, for each fixed $t\geq0$. So, by
standard arguments (see \cite{McCann}), one can show that the functional
$\mathcal{U}(t,\cdot)$ is lower semicontinuous with respect to the weak
topology, for each fixed $t\geq0$. Thus, $\mathcal{U}(t,\cdot)$ verifies
\textbf{E1}.

Let $\mu\in\mathscr{P}_{2,ac}(\mathbb{R}^{d})$ be such that $d\mu=\rho dx$ and
$\mathcal{U}(0,\mu)<\infty$. We have
\[
U(t,\rho(x))\leq\left(  \int_{0}^{t}{a(r)\ dr}\right)  U^{-}(0,\rho
(x))+U(0,\rho(x)),
\]
and then $\mathcal{U}(t,\mu)<\infty$ for all $t>0$. On the other hand, if
$d\mu=\rho dx$ is such that $\mathcal{U}(t,\mu)<+\infty$ for all $t>0$, then,
by substituting $z=\rho(x)$ in \eqref{eq:super-lineality}, we get%

\begin{equation}
\left(  1-\int_{0}^{t}{A(r)\ dr}\right)  U^{+}(0,\rho(x))-U^{-}(0,\rho(x))\leq
U(t,\rho(x)). \label{aux-internal-1}%
\end{equation}
It follows by integrating (\ref{aux-internal-1}) that $\mathcal{U}%
(0,\mu)<\infty.$ So, $\mathcal{U}(t,\mu)$ verifies \textbf{E2}.

Denote $\mathit{Dom}(\mathcal{U}(t,\cdot))=\mathbf{D}\subset\mathscr{P}_{2,ac}%
(\mathbb{R}^{d}).$ Since $U^{-}(0,0)=U^{+}(0,0)=0$, note that
\[
U(t,0)=\int_{0}^{t}{\frac{\partial U}{\partial t}(r,0)\ dr}=0,
\]
and then $\mathbf{D}$ is nonempty. For $s<t$ and $\mu\in\mathbf{D}$ with
$d\mu=\rho dx$, we have
\[
U(t,\rho(x))-U(s,\rho(x))\leq\left(  \int_{s}^{t}{a(r)\ dr}\right)  (c_{1}%
\rho(x)+c_{2}\rho(x)^{\alpha}).
\]
The same arguments used in \eqref{eq:under-bound1} lead us to
\begin{equation}
\mathcal{U}(t,\mu)-\mathcal{U}(s,\mu)\leq\left(  \int_{s}^{t}{a(r)\ dr}%
\right)  \left(  c_{1}+c_{2}C_{\alpha}(1+M_{2}(\mu))\right)  \text{, for all
}0\leq s<t. \label{eq:otto-inequality-generalized}%
\end{equation}
We are going to use \eqref{eq:otto-inequality-generalized} as a substitute for
the condition \textbf{E3}. Also, \textbf{E4 }follows from
\eqref{eq:otto-inequality-generalized}. In fact,
\[
\mathcal{U}(t,\mu)+\frac{\mathbf{d}_{2}^{2}(\mu,\delta_{0})}{2\tau^{\ast}}%
\geq-(c_{1}+c_{2}C_{\alpha}(1+M_{2}(\mu))^{\alpha})+\frac{M_{2}(\mu)}%
{2\tau^{\ast}}.
\]
Now, it is easy to see that for $\tau^{\ast}>0$ small enough the last
expression is bounded from below, as desired. Note that we have \textbf{E5}
with $\lambda\equiv0$ because \textbf{U3} implies that $\mathcal{U}(t,\cdot)$
is convex along of generalized geodesics.

Let us remark that Lemmas \ref{lemma:existence-minimizer} and
\ref{lem:continuity-minimizer} can be proved by proceeding as in Section
\ref{sect:construc-and-propert} (and using \eqref{eq:super-lineality} for
Lemma \ref{lem:continuity-minimizer}). In order to recover the
differentiability property in Proposition
\ref{prop:differentiability-property1}, we recall the notation
\[
\mathscr{E}_{t,\tau}(\mu)=\inf_{\nu\in\mathscr{P}_{2}(\mathbb{R}^{d})}\left\{
\frac{\mathbf{d}_{2}^{2}(\mu,\nu)}{2\tau}+\mathcal{U}(t,\nu)\right\}
=\frac{\mathbf{d}_{2}^{2}(\mu,\mu_{\tau}^{t})}{2\tau}+\mathcal{U}(t,\mu_{\tau
}^{t}).
\]
Then, by taking $\tau_{0}<\tau_{1}$ and $\mu\in\mathbf{D,}$ we have
\begin{align}
\mathscr{E}_{t+\tau_{1},\tau_{1}}(\mu)-\mathscr{E}_{t+\tau_{0},\tau_{0}}(\mu)
&  \leq\mathcal{U}(t+\tau_{1},\mu_{\tau_{0}}^{t+\tau_{0}})-\mathcal{U}%
(t+\tau_{0},\mu_{\tau_{0}}^{t+\tau_{0}})\nonumber\\
&  \ +\frac{\tau_{0}-\tau_{1}}{2\tau_{0}\tau_{1}}\mathbf{d}_{2}^{2}(\mu
,\mu_{\tau_{0}}^{t+\tau_{0}}). \label{eq:est-dif-intern-en1}%
\end{align}
Denoting $d\mu_{\tau}^{t+\tau}=\rho_{\tau}^{t+\tau}dx,$ we can estimate
\begin{align*}
\mathcal{U}(t+\tau_{1},\mu_{\tau_{0}}^{t+\tau_{0}})-\mathcal{U}(t+\tau_{0}%
,\mu_{\tau_{0}}^{t+\tau_{0}})  &  =\int_{\mathbb{R}^{d}}{\int_{t+\tau_{0}%
}^{t+\tau_{1}}{\frac{\partial U}{\partial r}(r,\rho_{\tau_{0}}^{t+\tau_{0}%
}(x))\ dr}\ dx}\\
&  \leq\int_{t+\tau_{0}}^{t+\tau_{1}}{a(r)\ dr}\int_{\mathbb{R}^{d}}%
{U^{-}(0,\rho_{\tau_{0}}^{t+\tau_{0}}(x))\ dx}.
\end{align*}
The last integral on $\mathbb{R}^{d}$ is uniformly bounded in $\tau_{0}$ on
compact sets of $(0,\tau^{\ast}]$. Replacing the roles of $\tau_{0}$ and
$\tau_{1},$ we get
\begin{align}
\mathscr{E}_{t+\tau_{1},\tau_{1}}(\mu)-\mathscr{E}_{t+\tau_{0},\tau_{0}}(\mu)
&  \geq\mathcal{U}(t+\tau_{1},\mu_{\tau_{1}}^{t+\tau_{1}})-\mathcal{U}%
(t+\tau_{0},\mu_{\tau_{1}}^{t+\tau_{1}})\nonumber\\
&  \ +\frac{\tau_{0}-\tau_{1}}{2\tau_{0}\tau_{1}}\mathbf{d}_{2}^{2}(\mu
,\mu_{\tau_{2}}^{t+\tau_{1}}) \label{eq:est-dif-intern-en2}%
\end{align}
and
\[
\mathcal{U}(t+\tau_{1},\mu_{\tau_{1}}^{t+\tau_{1}})-\mathcal{U}(t+\tau_{0}%
,\mu_{\tau_{1}}^{t+\tau_{1}})\geq-\int_{t+\tau_{0}}^{t+\tau^{1}}{A(r)\ dr}%
\int_{\mathbb{R}^{d}}{U^{+}(0,\rho_{\tau_{1}}^{t+\tau_{1}})\ dx.}%
\]
Now we need a uniform estimate for $\int_{\mathbb{R}^{d}}{U^{+}(0,\rho
_{\tau_{1}}^{t+\tau_{1}}(x))dx}$. Substituting $t=t+\tau_{1}$ and
$z=\rho_{\tau_{1}}^{t+\tau_{1}}(x)$ in \eqref{eq:super-lineality}, and
afterwards integrating it, we arrive at
\begin{equation}
\left(  1-\Vert A\Vert_{1}\right)  \int_{\mathbb{R}^{d}}{U^{+}(0,\rho
_{\tau_{1}}^{t+\tau_{1}}(x))\ dx}\leq\int_{\mathbb{R}^{d}}{U^{-}(0,\rho
_{\tau_{1}}^{t+\tau_{1}}(x))\ dx}+\mathscr{E}_{t+\tau_{1},\tau_{1}}(\mu).
\label{eq:est-dif-intern-en3}%
\end{equation}
The first term in the right hand side of \eqref{eq:est-dif-intern-en3} is locally
uniformly bounded in $(0,\tau^{\ast}]$. By the continuity of the map
$\tau\rightarrow\mathscr{E}_{t+\tau,\tau}(\mu),$ the second term also verifies
so. Therefore, we conclude that the function $\tau\rightarrow
\mathscr{E}_{t+\tau,\tau}(\mu)$ is absolutely continuous in each compact
subinterval of $(0,\tau^{\ast}]$. Now a version of the dominated convergence
theorem leads us to the formula

\begin{equation}
\frac{d}{d\tau}\mathscr{E}_{t+\tau,\tau}(\mu)=\int_{\mathbb{R}^{d}}%
{\frac{\partial U}{\partial r}(t+\tau,\rho_{\tau}^{t+\tau}(x))\ dx}%
-\frac{\mathbf{d}_{2}^{2}(\mu,\mu_{\tau}^{t+\tau})}{2\tau^{2}},
\label{eq:formula-deriv-internal-energy}%
\end{equation}
for each differentiability point $\tau\in(0,\tau^{\ast}].$ The identity above
implies the integral equality (\ref{eq:integral-identity}) in Corollary
\ref{cor:integral-identity}.

In the sequel, we sketch the proof of Lemma \ref{lemma:limit-sec-discret} in
the case of this present section. Recalling the notation for the discrete
solution in \eqref{eq:minimization-problem}, we have
\begin{align*}
\frac{1}{2}(M_{2}(U_{\boldsymbol{\tau}}^{n})-M_{2}(U_{\boldsymbol{\tau}}%
^{0}))  &  \leq\sum_{j=1}^{n}\frac{1}{2}(M_{2}(U_{\boldsymbol{\tau}}%
^{j})-M_{2}(U_{\boldsymbol{\tau}}^{j-1}))\\
&  \leq\sum_{j=1}^{n}\frac{1}{2}(\tau^{\ast}\frac{\mathbf{d}_{2}%
^{2}(U_{\boldsymbol{\tau}}^{j},U_{\boldsymbol{\tau}}^{j-1})}{2\tau_{j}}%
+2\tau_{j}\frac{M_{2}(U_{\boldsymbol{\tau}}^{j})}{\tau^{\ast}})\\
&  \leq\frac{\tau^{\ast}}{2}(\mathcal{U}(0,U_{\boldsymbol{\tau}}%
^{0})-\mathcal{U}(t_{\boldsymbol{\tau}}^{n},U_{\boldsymbol{\tau}}^{n}%
))+\sum_{j=1}^{n}\frac{\tau_{j}}{\tau^{\ast}}M_{2}(U_{\boldsymbol{\tau}}%
^{j})\\
&  \ +\frac{\tau^{\ast}}{2}\sum_{j=1}^{n}(\mathcal{U}(t_{\boldsymbol{\tau}}%
^{j},U_{\boldsymbol{\tau}}^{j-1})-\mathcal{U}(t_{\boldsymbol{\tau}}%
^{j-1},U_{\boldsymbol{\tau}}^{j-1})).
\end{align*}
Using the above estimate and \eqref{eq:otto-inequality-generalized}, we can
proceed as in Lemma \ref{lemma:limit-sec-discret} and reobtain the conclusions
of this lemma for the functional $\mathcal{U}(t,\mu).$

Now we deal with the convergence of the approximate solutions. In comparison
with subsection 4.2, there is only a new term that reads as%
\[
\int_{0}^{t}{(1-l_{\boldsymbol{\tau}}(t))\left[  \mathcal{U}(\mathcal{T}%
_{\boldsymbol{\tau}}(t),\underline{U}_{\boldsymbol{\tau}}(t))-\mathcal{U}%
(\mathcal{T}_{\boldsymbol{\eta}}(t),\underline{U}_{\boldsymbol{\tau}}%
(t))\right]  \ dt}.
\]
Notice that it is necessary to consider only the case $\mathcal{T}%
_{\boldsymbol{\tau}}(t)<\mathcal{T}_{\boldsymbol{\eta}}(t)$. So, we have that
\begin{equation}
\mathcal{U}(\mathcal{T}_{\boldsymbol{\tau}}(t),\underline{U}%
_{\boldsymbol{\tau}}(t))-\mathcal{U}(\mathcal{T}_{\boldsymbol{\eta}}%
(t),\underline{U}_{\boldsymbol{\tau}}(t))\leq\left(  \int_{\mathcal{T}%
_{\boldsymbol{\tau}}(t)}^{\mathcal{T}_{\boldsymbol{\eta}}(t)}{A(r)\ dr}%
\right)  \int_{\mathbb{R}^{d}}{U^{+}(0,\underline{U}_{\boldsymbol{\tau}}%
(t,x))\ dx}. \label{aux-discrete-1}%
\end{equation}
Using \eqref{eq:super-lineality} we can estimate the integral over
$\mathbb{R}^{d}$ in (\ref{aux-discrete-1}) locally uniformly in $[0,\infty)$.
Now, by replacing the function $\beta$ by $a$ or $A,$ one can repeat the same
arguments in the proof of Proposition \ref{prop:residuo-G}, obtain the
estimate (\ref{eq:proposition-est-G}) and afterwards the convergence of the
approximate solutions \eqref{eq:sol-interpolante1} to a curve $\mu
:[0,\infty)\rightarrow\mathscr{P}_{2}(\mathbb{R}^{d})$.

In what follows, we sketch the main arguments for the convergence of the De Giorgi
interpolation (see Proposition \ref{prop:DiGiorgi-interp-converg}).

Let $\delta=t-t_{\boldsymbol{\tau}}^{n-1}$, for $t\in(t_{\boldsymbol{\tau}}%
^{n-1},t_{\boldsymbol{\tau}}^{n}]$. Using the minimizer property of
$\widetilde{U}_{\boldsymbol{\tau}}$ and $\overline{U}_{\boldsymbol{\tau}}$, we
can obtain
\begin{align}
\frac{\tau_{n}-\delta}{2\tau_{n}\delta}\mathbf{d}_{2}^{2}(\underline
{U}_{\boldsymbol{\tau}}(t),  &  \widetilde{U}_{\boldsymbol{\tau}}%
(t))+\mathcal{U}(t,\widetilde{U}_{\boldsymbol{\tau}}(t))-\mathcal{U}%
(t_{\boldsymbol{\tau}}^{n},\widetilde{U}_{\boldsymbol{\tau}}(t))\nonumber\\
&  \leq\frac{\tau_{n}-\delta}{2\tau_{n}\delta}\mathbf{d}_{2}^{2}(\underline
{U}_{\boldsymbol{\tau}}(t),\overline{U}_{\boldsymbol{\tau}}(t))+\mathcal{U}%
(t,\overline{U}_{\boldsymbol{\tau}}(t))-\mathcal{U}(t_{\boldsymbol{\tau}}%
^{n},\overline{U}_{\boldsymbol{\tau}}(t)), \label{aux-gen-1}%
\end{align}
and so
\begin{align}
\frac{\tau_{n}-\delta}{2\tau_{n}\delta}\mathbf{d}_{2}^{2}(\underline
{U}_{\boldsymbol{\tau}}(t),\widetilde{U}_{\boldsymbol{\tau}}(t))  &  \leq
\frac{\tau_{n}-\delta}{2\tau_{n}\delta}\mathbf{d}_{2}^{2}(\underline
{U}_{\boldsymbol{\tau}}(t),\overline{U}_{\boldsymbol{\tau}}(t))\nonumber\\
&  \ +\int_{t}^{t_{\boldsymbol{\tau}}^{n}}{a(r)\ dr}\int_{\mathbb{R}^{d}%
}{U^{-}(0,\widetilde{U}_{\boldsymbol{\tau}}(t,x))\ dx}\nonumber\\
&  \ +\int_{t}^{t_{\boldsymbol{\tau}}^{n}}{A(r)\ dr}\int_{\mathbb{R}^{d}%
}{U^{+}(0,\overline{U}_{\boldsymbol{\tau}}(t,x))\ dx}. \label{aux-gen-3}%
\end{align}
By using these estimates, one can obtain the conclusions of Proposition
\ref{prop:DiGiorgi-interp-converg}. In order to recover some properties in Theorem
\ref{theor:sol-flux-grad}, note that the inequality $\geq$ in
\eqref{eq:energy-identity} follows from the same arguments and Fatou's Lemma. The absolutely continuity of the map $t\rightarrow\mathcal{U}(t,\mu(t))$ in
each bounded interval of $[0,\infty)$ is a consequence of
\eqref{eq:formula-inclinacao} with $\lambda(t)\equiv0$ and the estimate
(\ref{eq:super-lineality}).

Although we are not able to obtain the reverse inequality $\leq$ in
\eqref{eq:energy-identity}, and so the energy identity
\eqref{eq:energy-identity} (see Remark \ref{remark:flexibility-2} below), it is not hard to show the estimate
\begin{align*}
\mathcal{U}(t,\mu(t))-\mathcal{U}(s,\mu(s))  &  \leq\int_{s}^{t}%
{\int_{\mathbb{R}^{d}}{\frac{\partial U}{\partial t}(r,\mu(r))\ dx}\ dr}%
-\frac{1}{2}\int_{s}^{t}{|\mu^{\prime}|^{2}(r)\ dr}\\
&  \ -\frac{1}{2}\int_{s}^{t}{|\partial\mathcal{U}(r)|^{2}(\mu(r))\ dr}.
\end{align*}

Even without the energy identity and minimal selection, it is possible to show
that the approximate solutions converge to a distributional solution $\mu$.
Denote $P(t,z)=z\frac{\partial U}{\partial z}(t,z)-U(t,z)$. We summarize the
results for the functional (\ref{eq:general-internal-energy}) in the theorem below.

\begin{teor}
\label{teor:internal energy} Consider the internal energy functional defined
in \eqref{eq:general-internal-energy} and assume \textbf{U1} to \textbf{U3}.
Then the curve $\mu:[0,\infty)\rightarrow\mathscr{P}_{2}(\mathbb{R}^{d}),$
$d\mu_{t}=\rho(x,t)dx$, given by Theorem \ref{teor:convergence-scheme-discret}
is a distributional solution of the equation
\begin{equation}
\partial_{t}\rho-\nabla_{x}\cdot\left(  \nabla_{x}P(t,\rho(t,x))\right)  =0
\label{eq:difussion-general}%
\end{equation}
with initial condition $\mu_{0}\in\mathscr{P}_{2}(\mathbb{R}^{d})$. Moreover,
the contraction property (\ref{eq:contraction-property}) holds.
\end{teor}

\begin{proof}
For simplicity take a uniform step size $\tau>0$, consider the partition $\{0<\tau<2\tau<3\tau<\cdots\}$, and choose
$\xi\in C^{\infty}_0(\mathbb{R}^{d} ;\mathbb{R}^d)$. Consider also the flow $\Phi_{\delta}$ associated to the field $\xi$, i.e.
\begin{equation}
(\Phi_{\delta})'=\xi(\Phi).\label{aux-ineq-1}
\end{equation}
Then, by the minimizer property of $U_{\tau}^n$, we have
\begin{eqnarray}
\mathcal{U}(n\tau,U_{\delta})+\frac{\mathbf{d}_2^2(U_{\tau}^{n-1},U_{\delta})}{\tau}-\mathcal{U}(n\tau,U_{\tau}^n)
-\frac{\mathbf{d}_2^2(U_{\tau}^{n-1},U_{\tau}^n)}{\tau}\geq0,\label{aux-ineq-2}
\end{eqnarray}
where $U_{\delta}=\Phi_{\delta}\#U_{\tau}^n$ is the push-forward of $U_{\tau}^n$ via $\Phi_{\delta}$. Then, by standard arguments (see e.g. \cite{Jordan-Otto}), it follows that
\begin{equation*}
\lim_{\delta\to0}\frac{\mathcal{U}(n\tau,U_{\delta})-\mathcal{U}(n\tau,U_{\tau}^n)}{\delta}=\int_{\mathbb{R}^d}
{-P(n\tau,U_{\tau}^n(x))\text{div}\xi\ dx}
\end{equation*}
and
\begin{equation*}
\lim_{\delta\to0}\tau^{-1}\frac{\mathbf{d}_2^2(U_{\tau}^{n-1},U_{\delta})-\mathbf{d}_2^2(U_{\tau}^{n-1},U_{\tau}^n)}{\delta}
=\int_{\mathbb{R}^d\times\mathbb{R}^d}{\frac{(x-y)}{\tau}\cdot \xi(y)\ d\gamma(x,y)},
\end{equation*}
where $\gamma\in\Gamma(U_{\tau}^{n-1},U_{\tau}^{n})$ is an optimal plane for the transport from $U_{\tau}^{n-1}$ to
$U_{\tau}^{n}$. Changing $\xi$ by $-\xi$ in \eqref{aux-ineq-1} (by symmetry in \eqref{aux-ineq-2}) and taking $\xi=\nabla\zeta$, we obtain that
\begin{equation}
\label{eq:cont-aprox-internal}
\int_{\mathbb{R}^d\times\mathbb{R}^d}{\frac{(x-y)}{\tau}\cdot \nabla\zeta(y)\ d\gamma(x,y)}-\int_{\mathbb{R}^d}{P(n\tau,
U_{\tau}^n(x))\Delta\zeta\ dx}=0.
\end{equation}
Let us remark that the above calculations also allow to conclude that $P(n\tau,U_{\tau}^n)\in W^{1,2}(\mathbb{R}^d)$
is bounded uniformly. Thus, we can use an argument of weak convergence and estimates as in Lemma \ref{lemma:limit-sec-discret} in order to obtain that the curve $\mu:[0,\infty)\to\mathscr{P}_2(\mathbb{R}^d)$ solves \eqref{eq:difussion-general} in the distributional
sense.
\end{proof}

\begin{rem}
\label{remark:flexibility} The conditions \textbf{U1} to \textbf{U3} work well
if we consider a functional as being the sum of the internal energy and
another functional as in two previous subsections. In the present subsection,
we have preferred to consider only the internal energy for the sake of simplicity.
\end{rem}

\begin{rem}
\label{remark:flexibility-2} Let us observe that the energy identity was not
obtained in Theorem \ref{teor:internal energy}. The reason is that, in order
to obtain such property in this general case, it would be necessary to handle
the limit
\begin{equation}
\lim_{t\rightarrow t_{0}}\frac{\mathcal{U}(t,\mu(t))-\mathcal{U}(t_{0}%
,\mu(t))}{t-t_{0}}. \label{aux-monge}%
\end{equation}
By making a change of variable (see \cite[Theorem 4.8]{Villani}), the calculus
of \eqref{aux-monge} is related to stability results for the Monge-Amp\`{e}re
equation. However, as far as we know, such results are available in the
literature (see \cite{DePhili-Fig-2}) under restrictions stronger than the
ones that we have in our context.
\end{rem}

\end{document}